\documentclass[12pt]{amsart}

\usepackage[margin=1in]{geometry}

\usepackage{microtype}
\usepackage{amsmath}
\usepackage[all]{xy}
\usepackage{amssymb}
\usepackage{enumitem}

\usepackage{amsthm}
\usepackage{thmtools}
\declaretheorem[numberwithin=section]{theorem}
\declaretheorem[sibling=theorem]{corollary}
\declaretheorem[sibling=theorem]{lemma}
\declaretheorem[sibling=theorem]{proposition}
\declaretheorem[sibling=theorem]{claim}
\declaretheorem[style=definition]{definition}
\declaretheorem[style=remark,numberwithin=section]{remark}
\declaretheorem[style=definition,numberwithin=section]{example}
\declaretheorem[style=definition,numberwithin=section]{notation}
\declaretheorem[style=definition,numberwithin=section]{convention}
\declaretheorem[style=definition,numberwithin=section]{question}
\declaretheorem{conjecture}
\newtheorem*{theorem*}{Theorem}
\newtheorem*{lemma*}{Lemma}
\newtheorem*{proposition*}{Proposition}
\newtheorem*{conjecture*}{Conjecture}
\newtheorem*{corollary*}{Corollary}
\newtheorem*{problem*}{Problem}

\renewcommand{\P}{\mathbb P}
\newcommand{\R}{\mathbb R}
\newcommand{\Z}{\mathbb Z}

\newcommand{\bv}{\mathbf v}
\newcommand{\cO}{\mathcal O}
\DeclareMathOperator{\Eff}{Eff}

\DeclareMathOperator{\mult}{mult}

\DeclareMathOperator{\Sec}{Sec}
\DeclareMathOperator{\Effb}{\smash{\overline{\Eff}}\vphantom{\Eff}}
\DeclareMathOperator{\CEffb}{C\!\Effb}
\DeclareMathOperator{\ev}{ev}
\DeclareMathOperator{\Bl}{Bl}

\newcommand{\abs}[1]{\left\lvert #1 \right\rvert}
\newcommand{\set}[1]{\left\{ #1 \right\}}

\begin{document}

\title[Effective cycles on blow-ups of $\mathbb P^n$]{Effective cones of cycles \\ on blow-ups of projective space}

\author[I. Coskun]{Izzet Coskun}
\address{Department of Mathematics, Statistics and CS \\ University of Illinois at Chicago, Chicago, IL 60607}
\email{coskun@math.uic.edu}

\author[J. Lesieutre]{John Lesieutre}
\address{Department of Mathematics, Statistics and CS \\ University of Illinois at Chicago, Chicago, IL 60607}
\email{jdl@uic.edu}

\author[J. C. Ottem]{John Christian Ottem}
\address{Department of Mathematics, University of Oslo, Box 1053, Blindern, 0316 Oslo, Norway}
\email{johnco@math.uio.no}

\thanks{During the preparation of this article the first author was partially supported by the NSF CAREER grant DMS-0950951535 and NSF grant DMS-1500031, the second author was partially supported by an NSF RTG grant, and the third author was supported by RCN grant 250104.}


\begin{abstract}
In this paper, we study the cones of higher codimension (pseudo)effective cycles on point blow-ups of projective space. We determine bounds on the number of points for which these cones are generated by the classes of linear cycles, and for which these cones are finitely generated.  Surprisingly, we discover that for (very) general points, the higher codimension cones behave better than the cones of divisors. For example, for the  blow-up  $X_r^n$ of $\P^n$, $n>4$, at $r$ very general points, the cone of divisors is not finitely generated as soon as $r> n+3$, whereas the cone of curves is generated by the classes of lines if $r \leq 2^n$. In fact, if $X_r^n$ is a Mori Dream Space then all the effective cones of cycles on $X_r^n$ are finitely generated.
\end{abstract}

\maketitle

\section{Introduction}

In recent years, the theory of cones of cycles of higher codimension has been the subject of increasing attention \cite{ChenCoskun}, \cite{DELV}, \cite{DebarreJiangVoisin}, \cite{FulgerLehmannCone}, 
\cite{FulgerLehmannKernel}.  However, these cones have been computed only for a very small number of examples, mainly because  the current theory is hard to apply in practice. The goal of this paper is to provide some much-needed examples.

Let $\Gamma$ be a set of $r$ distinct points on $\P^n$.  Let $X_{\Gamma}^n$ denote the blow-up of $\P^n$ along  $\Gamma$. When $\Gamma$ is a set of $r$ very general points, we denote $X_{\Gamma}^n$ by $X_r^n$. For a smooth variety $Y$, we write $\Effb^k(Y)$ for the pseudoeffective cone of codimension-$k$ cycles on $Y$, and $\Effb_k(Y)$ for the pseudoeffective cone of dimension-\(k\) cycles.  In this paper, we study the cones $\Effb_k(X_{\Gamma}^n)$ when the points of $\Gamma$ are either in linearly general or very general position.  We also investigate the cones when $\Gamma$ contains points in certain special configurations.

Cones of positive divisors on $X_{\Gamma}^n$ provide an important source of examples in the study of positivity.  These cones are particularly attractive since they have concrete interpretations in terms of subvarieties of projective space, yet still have very complicated structure.  However, even the cones of divisors on blow-ups of \(\P^2\) at \(10\) or more points are far from well-understood, and several basic questions remain open, including the Nagata \cite{Nagata} and Segre--Harbourne--Gimigliano--Hirschowitz (SHGH) conjectures \cite{Gimigliano}, \cite{Harbourne}, \cite{Hirschowitz}.  We expect the cones of higher codimension cycles on $X_{\Gamma}^n$ to be an equally rich source of examples. 

Surprisingly, these cones are simpler than one might expect.  Effective cones of low-dimensional cycles are generated by the classes of linear spaces for \(r\) well into the range for which $X_r^n$ ceases to be a Mori Dream Space. For example,  $\Effb_1(X_r^n)$ is generated by classes of lines for \(r \leq 2^n\) even though $\Effb^1(X_r^n)$ is not finitely generated for $r\geq n+4$ when $n \geq 5$.   We now describe our results in greater detail.

\begin{definition}
We say that $\Effb_k(X_{\Gamma}^n)$ is \emph{linearly generated} if it is the cone spanned by the classes of $k$-dimensional linear spaces in the exceptional divisors and the strict transforms of  $k$-dimensional  linear subspaces of \(\P^n\), possibly passing through the points of \(\Gamma\).  We say  $\Effb_k(X_{\Gamma}^n)$ is \emph{finitely generated} if it is a rational polyhedral cone.
\end{definition}

\begin{theorem*}[\ref{lingen}]
Let $\Gamma$ be a set of $r$ points in $\P^n$ in linearly general position.  If $r \leq \max\left(n+2, n + \frac{n}{k}  \right)$,  then $\Effb_k(X_{\Gamma}^n)$ is linearly generated.
\end{theorem*}

There exist configurations of $2n+2-k$ points in linearly general position in $\P^n$ for which $\Effb_k(X_{\Gamma}^n)$ is not linearly generated (see Example \ref{ex-linearSharp}). In particular, Theorem \ref{lingen} is sharp for $1$-cycles.
We expect that this bound can be improved to $r \leq 2n+1-k$, and prove this in the case that  \(\Gamma\) is a very general configuration of points (Theorem \ref{thm-verygenerallingen}). We obtain the following consequence.

\begin{corollary*}[\ref{cor-MDS}]
If $X^n_r$ is a Mori Dream Space, then $\Effb_k(X^n_r)$ is finitely generated.
\end{corollary*}
In general, Mori Dream Spaces may have effective cones of intermediate dimensional cycles which are not finitely generated; the corollary shows that this does not happen for blow-ups of \(\P^n\). A good example is \cite[Example 6.10]{DELV} attributed to Tschinkel. Let $X_b$ be the blow-up  of $\P^4$ along a smooth quartic K3 surface $Y_b \subset \P^3 \subset  \P^4$. Then $X_b$ is Fano, hence, by \cite{BCHM}, a Mori Dream Space. On the other hand, $\Effb_2(X_b)$ has infinitely many extremal rays when $\Effb_1(Y_b)$ does. Quartic K3 surfaces may have infinitely many $(-2)$-curves or even a round cone of curves. This example also shows that the property of having finitely generated  higher codimension cones can fail countably many times in a family.   

The bounds can be exponentially improved (at least for $1$-cycles) if we assume that $\Gamma$ is a set of very general points.  
\begin{proposition*}[\ref{prop-verygeneralCurve}]
The cone $\Effb_1(X_{r}^n)$ is linearly generated if and only if $r \leq 2^n$.
\end{proposition*}

As a consequence of Proposition \ref{prop-verygeneralCurve}, we conclude that $\Effb_k(X_r^n)$ is not linearly generated if $r  \geq 2^{n-k+1} + k$ (Corollary \ref{cor-linear}).  This specializes to the fact that the cone of divisors of $X_r^n$ is not linearly generated as soon as $r > n+2$ (see Theorem \ref{divisorcase}). 

Work of Mukai \cite{Mukai} shows that the cone of divisors of $X_r^n$ is not finitely generated if  \(r \geq n+4\) and $n\geq 5$ (one needs $r\geq 8$ for $n=2,4$ and $r \geq 7$ for $n=3$).  Mukai explicitly constructs  infinitely many extremal divisors on \(\Effb^1(X_{r}^n)\), as the orbit of one of the exceptional divisors under the action of Cremona transformations.  However, in higher codimensions it is more difficult to prove that the corresponding cones  become infinite.

Many questions about cones of higher codimension cycles appear to be intractable, quickly reducing to difficult questions about cones of divisors.  For example, the interesting part of the cone of curves of \(\P^3\) blown up at \(9\) points is given by curves lying on the unique quadric \(Q\) through the \(9\) points.  The blow-up of \(Q\) is isomorphic to the blow-up of \(\P^2\) at \(10\) points, and the curves which are extremal on \(X^3_9\) are certain \(K_Q\)-positive ones contained in \(Q\).  Hence understanding \(\Effb_1(X^3_9)\) requires understanding the \(K_{X^2_{10}}\)-positive part of \(\Effb_1(X^2_{10})\), running immediately into the SHGH conjecture (see Conjecture~\ref{shghstatement}). We are able to show this non-finiteness only for cones of codimension-\(2\) cycles, and then assuming the SHGH conjecture on the cone of curves of \(\P^2\) blown up at \(10\) points.    
\begin{corollary*}[\ref{cor-codim2infinite}]
Assume the SHGH conjecture holds for blow-ups of \(\P^2\) at \(10\) points.  Then \(\Effb^2(X_{r}^n)\) is not finitely generated if $r \geq n+6$ and $n \geq 3$.
\end{corollary*}

Finally, in the last section, we compute $\Effb_k(X_{\Gamma}^n)$ when $\Gamma$ is a set of points in certain special positions. Using these computations, we show that linear and finite generation of $\Effb_k(X_{\Gamma}^n)$ are neither open nor closed in families (see Corollary \ref{cor-closed} and Corollary \ref{cor-open}). This generalizes analogous jumping behavior exhibited for divisors and Mori Dream Spaces to all codimensions.

\subsection*{The organization of the paper} In \S \ref{sec-prelim}, we collect basic facts concerning the cohomology of $X_{\Gamma}^n$, cones of divisors, the action of Cremona transformations, and some preliminary lemmas. In \S \ref{sec-lingen}, we prove Theorem \ref{lingen} and study the linear generation of the cones $\Effb_k(X_{\Gamma}^n)$ when $\Gamma$ is a linearly general set of points. In \S \ref{sec-nonlingen}, we study the linear generation of the cones $\Effb_k(X_{r}^n)$. In \S \ref{sec-nonfingen}, we prove that $\Effb^2(X_{r}^n)$ is not finitely generated for $r \geq n+6$ assuming the SHGH conjecture. In \S \ref{sec-special}, we discuss the cones $X_{\Gamma}^n$ when $\Gamma$ contains points in certain special configurations and study the variation of $\Effb_k(X_{\Gamma}^n)$ in families. 

\subsection*{Acknowledgements} We would like to thank Dawei Chen, Lawrence Ein, Mihai Fulger, Joe Harris, Brian Lehmann, Kristian Ranestad, and Kevin Tucker for helpful discussions on cones of higher codimension cycles.

\section{Preliminaries}\label{sec-prelim}
In this section, we recall basic facts about the cohomology of $X_{\Gamma}^n$ and cones of codimension-$1$ cycles.  We will work over the complex numbers $\mathbb{C}$. 

\subsection*{The cohomology of $X_{\Gamma}^n$}
Let $\Gamma$ be a set of $r$ points $p_1, \dots, p_r$ in $\P^n$, and let $$\pi: X_{\Gamma}^n  = \Bl_{\Gamma} \P^n \rightarrow \P^n $$ denote the blow-up of $\P^n$ along $\Gamma$. Let $H$ denote the pullback of the hyperplane class and let $E_i$ denote the class of the exceptional divisor over $p_i$.  The exceptional divisor $E_i$ is isomorphic to $\P^{n-1}$ and $\cO_{E_i}(E_i) \cong \cO_{\P^{n-1}}(-1)$. Consequently, we have the following intersection formulas:
$$H^n = (-1)^{n-1} E_i^n=1, \quad H \cdot E_i = 0, \quad E_i \cdot E_j =0,\ \  i\not=j.$$
\begin{notation}
In order to simplify notation, we make the convention that $H_k$ is the class of a $k$-dimensional linear space in $\P^n$ and  $E_{i,k}$ is the class of a $k$-dimensional linear space contained in the exceptional divisor $E_i$. We then have the relations $$H^{n-k} = H_k, \quad (-1)^{n-k+1} E_i^{n-k} = E_{i,k}, \quad E_i \cdot E_{i,k} = - E_{i, k-1}.$$
\end{notation}

On $X_{\Gamma}^n$ homological, numerical and rational equivalence coincide.
For $0<k<n$, we write \(N_k(X_{\Gamma}^n)\) for the \(\R\)-vector space of \(k\)-dimensional cycles on \(X_\Gamma^n\), modulo numerical equivalence. Dually, \(N^k(X_{\Gamma}^n)\) denotes the space of codimension-\(k\) cycles modulo numerical equivalence. They are both $(r+1)$-dimensional vector spaces.  

A class in \(N_k(X_\Gamma^n)\) is said to be pseudoeffective if it is the limit of classes of effective cycles. We write \(\Effb_k(X_{\Gamma}^n)\) for the closed convex cone in \(N_k(X_{\Gamma}^n)\) containing pseudoeffective classes.  If \(V\) is an (irreducible) \(k\)-dimensional subvariety of \(X_{\Gamma}^n\), we write \([V]\) for the class of \(V\) in \(N_k(X_{\Gamma}^n)\), although when confusion seems unlikely we omit the brackets.

A set of points in \(\P^n\) is said to be \emph{linearly general} if no \(k+2\) points are contained in a linear subspace \(\P^k \subset \P^n\) for \(1 \leq k \leq n-1\).  A claim holds for a \emph{very general} configuration of points if it holds for all points in the complement of a countable union of proper configurations of points.

\begin{convention}\label{conv-terrible}
It is occasionally useful to compare the cones $\Effb_k(X_{\Gamma}^n)$ and $\Effb_k(X_{\Delta}^m)$, where $X_{\Gamma}^n$ and $X_{\Delta}^m$ are the blow-ups of $\P^n$ and  $\P^m$ along sets of points $\Gamma$ and $\Delta$, respectively. If $n>k$, we can identify $N_k(X_{\Gamma}^n)$ with the  abstract vector space spanned by $H_k$ and $E_{i,k}$ for $1 \leq i \leq r$, irrespective of $n$ and $\Gamma$ provided that $\Gamma$ has cardinality $r$.  We can thus view the cones $\Effb_k(X_{\Gamma}^n)$ as cones in the same abstract vector space and compare the effective cones of different blow-ups after this identification. In the rest of the paper, we will do so without further comment.
\end{convention}

We will often use the following easy lemma implicitly.
\begin{lemma}\label{lem-ab}
Let $Y \subset X_{\Gamma}^n$ be a $k$-dimensional subvariety.
\begin{enumerate}
\item If $Y \subset E_i$ for some $1 \leq i \leq r$, then $[Y] = b_i E_{i,k}$ for  \(b_i>0\).
\item Otherwise, $[Y] = aH_k - \sum_{i=1}^r b_i E_{i,k}$ with $a \geq b_i \geq 0$.  The coefficient \(b_i\) is equal to the multiplicity of \(\pi(Y)\) in \(\P^n\) at the point \(p_i\).
\end{enumerate}
\end{lemma}
\begin{proof}
If $Y \subset E_i$, then $Y$ is a subvariety of $E_i \cong \P^{n-1}$. Hence, its class is a positive multiple of the class of a $k$-dimensional linear space. 
The linear system $H-E_i$ defines the projection from the point $p_i$ and is a base-point-free linear system. Hence,  the intersection of $k$ general members of $H-E_i$ with $Y$ is either empty or finitely many points. Therefore, $(H-E_i)^k\cdot [Y] = a - b_i \geq 0$. Similarly, the intersection $Y \cap E_i$ is a (possibly empty) effective cycle of dimension $k-1$ contained in $E_i$. Hence, by the first part of the lemma, $b_i \geq 0$.  That \(b_i\) in fact coincides with the multiplicity is~\cite[Cor.\ 6.7.1]{Fulton}.
\end{proof}

The cones $\Effb_k(X_{\Gamma}^n)$ satisfy a basic semicontinuity property under specialization.

\begin{lemma}
\label{multincrease}
Suppose that \(V \subset \P^n \times T\) is a closed subvariety, flat over \(T\), and let \(p : T \to \P^n\) be a section.  Then \(\mult_{p(t)}(V_t)\) is an upper semicontinuous function of \(T\).
\end{lemma}
\begin{proof}
It suffices to prove this in the case that \(T\) has dimension \(1\).  Let \(\pi : Y \to \P^n \times T\) be the blow-up along \(p(T)\), with exceptional divisor \(E\), and let \(\tilde{V}\) be the strict transform of \(V\) on \(Y\).  Since  \(\tilde{V}\) is irreducible and dominates \(T\), this family is flat.  The intersection of a flat family of cycles with a Cartier divisor is constant in \(t\)~\cite[Prop.\ 10.2.1]{Fulton}, and so \((-1)^{k} E^k \cdot V_t\) is independent of \(t\).

The general fiber \(\tilde{V}_t\) is irreducible, but a special fiber \(\tilde{V}_0\) may have additional components in the exceptional divisor \(E_0\).  Write \(\tilde{V}_0 = V^0 \bigcup \cup_i W_i\), where the \(W_i\) are contained in \(E_0\).  Then \((-1)^{k} E^k \cdot V^0 = \mult_{p(0)} V\).  The class \((-1)^{k} E_0^k\) is a hyperplane in \(E_0\), and so \((-1)^{k} E_0^k \cdot W_i \geq 0\).  This shows that \(\mult_{p(0)} V_0 \geq (-1)^{k} E^k \cdot \tilde{V}_t = \mult_{p(t)} V_t \), and so the multiplicity is upper semicontinuous.
\end{proof}

\begin{corollary}
\label{cor-semicont}
Let \(\Gamma\) be a configuration of \(r\) distinct points on \(\P^n\).  Then \(\Effb_k(X^n_r) \subseteq \Effb_k(X_\Gamma^n)\).
\end{corollary}
\begin{proof}
Let \(\Gamma_t\) be a very general one-parameter family of configurations of points in \(\P^n\) with  \(\Gamma_0 = \Gamma\).  If a \(k\)-cycle class \(W\) is effective for very general \(T\), then by a Hilbert scheme argument there exists a flat family \(V_t \subset \Bl_{\Gamma_t} \P^n\) over \(T\) with \([V_t]  = W\) for general \(T\).  Since the multiplicity of \(W_t\) can only increase at \(t=0\) by Lemma~\ref{multincrease}, the class \(W\) is also effective on \(X_\Gamma\).
\end{proof}

\subsection*{Cones} Taking cones will be a useful method to generate interesting cycles. Let $\Gamma'$ be a very general configuration of $r+1$ points $p_0', \dots, p_r'$ in $\P^{n+1}$. The projection of the points $p_1', \dots, p_r'$ from \(p_0^\prime\) is then a set $\Gamma$ of $r$ very general points $p_1, \dots, p_r$ in $\P^n$.  Suppose that \(V\) is a \(k\)-cycle on \(X^n_{\Gamma}\), with class \(aH_k - \sum_{i=1}^r b_i E_{i,k}\).  The image of \(V\) in \(\P^n\) has degree \(a\) and multiplicity \(b_i\) at the points of \(\Gamma\).  We may form the cone \(CV\) over \(V\) inside \(\P^{n+1}\) with vertex at $p_0'$.  This is a \((k+1)\)-dimensional variety, of degree \(a\).  It has multiplicity \(a\) at the cone point, and multiplicity \(b_i\) along the lines  spanned by $p_i$ and  $p_i'$ for $1 \leq i \leq r$.  In particular, the cycle \(CV\) has degree \(a\), and multiplicities \(a,b_1,\ldots,b_r\) at the points of \(\Gamma'\). Its proper transform has class $ a H_{k+1} -a E_{0,k+1} - \sum_{i=1}^r b_i E_{i, k+1}.$

We define a map \(C : N_k(X^n_r) \to N_{k+1}(X^{n+1}_{r+1})\) by \(C(H_k) = H_{k+1} - E_{0,k+1}\) and \(C(E_{i,k}) = E_{i,k+1}\).  With this definition, \(C([V])\) is the class of the cone over \(V\) with vertex $p_0'$, and so \(C(\Effb_k(X^n_r)) \subseteq \Effb_{k+1}(X^{n+1}_{r+1})\) with respect to the identification discussed in Convention \ref{conv-terrible}.

The following simple computation of a dual cone will be useful on a number of occasions.
\begin{lemma}
\label{maininequality}
Suppose that \(\bv = (a,b_1,\ldots,b_r) \in \Z^{r+1}\) is a vector satisfying
\begin{enumerate}
\item \(a, b_i \geq 0\);
\item \(a \geq b_i\) for every \(i\);
\item \(na \geq \sum_{i=1}^r b_i\).
\end{enumerate}
Then \(\bv\) is a positive linear combination of the vectors \(e_i\) (\(1 \leq i \leq r\)) and \(h_I = e_0 - \sum_{i \in I} e_i\) with \(\abs{I} = n\).
\end{lemma}

\begin{proof}
Note first that the vectors \(h_I = e_0 - \sum_{i \in I} e_i\) with \(\abs{I} < n\) are positive linear combinations of the given vectors.
We now proceed by induction on \(a\).  The case \(a=1\) is clear: since by (2) \(a \geq b_i\) for each \(i\), each \(b_i\) is either \(0\) or \(1\).  By (3) there are at most \(n\) nonzero \(b_i\), and the vector is of the form claimed.

Suppose that \(a > 1\).  Let \(j\) be the minimum of \(n\) and the number of nonzero \(b_i\), and let \(J\) the the set of  indices of $j$ nonzero $b_i$.  Then the vector $h_J$ is a nonnegative linear combination of the given vectors.  Set $v'= (a', b_i') = \bv-h_J$. If $j \geq n$, $v'$ still satisfies all of the inequalities in question since $a' = a -1$ and $\sum b_i' = \sum b_i - n$. If $j < n$, then in view of inequality (2), the inequality (3) can be improved to $j a \geq \sum_{i=1}^r b_i$. Then $v'$ satisfies these improved inequalities. This completes the proof by induction on $a$.    
\end{proof}

Lemma~\ref{maininequality} implies that the cone \(\Effb_k(X^n_r)\) is linearly generated if and only if the class \((k+1)H_{n-k} - \sum_{i=1}^r E_{i,n-k}\) is nef.

\subsection*{The codimension-$1$ cones and Cremona actions} 
We are primarily interested in the question of when the cones of cycles on \(X_{\Gamma}^n\) are linearly or finitely generated.  For cones of divisors, the answers to these questions were worked out by Castravet--Tevelev and Mukai~\cite{CastravetTevelev}, \cite{Mukai}.

\begin{theorem}[\cite{CastravetTevelev}, \cite{Mukai}]
\label{divisorcase}
Let $\Gamma$ be a set of $r$ very general points in $\P^n$. 
The cone \(\Effb^1(X_{\Gamma}^n)\) is linearly generated if and only if \(r \leq n+2\), and finitely generated if and only if
\begin{enumerate}
\item \(n=2\) and \(r \leq 8\),
\item \(n=3\) and \(r \leq 7\),
\item \(n=4\) and \(r \leq 8\),
\item \(n \geq 5\) and \(r \leq n+3\)
\end{enumerate}
\end{theorem}

The characterization of cases when the effective cone of divisors is finitely generated is based on the study of the action of Cremona transformations on the pseudoeffective cone.  The Coxeter group \(W\) corresponding to a \(T\)-shaped Dynkin diagram of type \(T_{2,n+1,r-n-1}\) acts on \(N^1(X_{\Gamma}^n)\)  and preserves the pseudoeffective cone \(\Effb^1(X_{\Gamma}^n)\). This is an infinite group if \(\frac{1}{2} + \frac{1}{n+1} + \frac{1}{r-n-1} < 1\), which happens as soon as $n\geq 5$ and $r \geq n+4$.  (When $n=2$ or $4$, we need $r \geq 9$, while when $n=3$, we require $r \geq 8$). The orbit of a single exceptional divisor class gives an infinite set divisors spanning other extremal rays. For details on this group action, we refer to~\cite{Dolgachev} (see also~\cite{Coble}).

Unfortunately, there does not seem to be a simple way to use the Cremona action to understand cones of cycles of higher codimension.  The standard Cremona involution acts on \(X^n_r\) by a map with codimension \(2\) indeterminacy, so it does not define an action preserving the cone \(\Effb^k(X^n_r)\) for any \(k > 1\).  For example, suppose that \(L\) is a line through two blown up points. The class of \(L\) defines an extremal ray on \(\Effb_1(X^n_r)\). The strict transform of \(L\) under a Cremona transformation centered at \(n+1\) other points is a rational normal curve in \(\P^n\) passing through \(n+3\) points.  If \(n \geq 3\), this is no longer an extremal ray on \(\Effb_1(X^n_r)\), since it is in the interior of the subcone generated by classes of lines through \(2\) of the \(n+3\) points.

One might attempt to construct interesting codimension-\(2\) cycles on \(X_n^r\)  by taking the intersections of a fixed divisor with an infinite sequence of \((-1)\)-divisors (i.e.\ divisors in the orbit of \(E_i\) under the action of \(W\)) of increasing degree.  However, the next lemma shows that the intersection of a \((-1)\)-divisor with any other effective divisor on \(X^n_r\) is in the span of the classes of codimension-\(2\) linear cycles.

\begin{lemma}\label{lem-cremona}
Suppose that \(D_1\) is a \((-1)\)-divisor and that \(D_2\) is an irreducible effective divisor distinct from \(D_1\).  Then \([D_1 \cap D_2]\) is in the span of linear codimension-\(2\) cycles.
\end{lemma}
\begin{proof}
Consider the pairing on \(N^1(X^n_r)\) defined by \((H, H) = n-1\), \((H , E_i) = 0\), \((E_i , E_i) = -1\), and \((E_i , E_j) = 0\) if \(i \neq j\).  This pairing is invariant under the action of \(W\) on \(N^1(X^n_r)\)~\cite{Mukai}, \cite{Dolgachev}.

We first show that \((D_1 , D_2) \geq 0\).  Since the pairing (,) is invariant under the action of \(W\) on \(N^1(X^n_r)\), we may apply a suitable element of \(W\) and assume that \(D_1 = E_1\) is an exceptional divisor.
If \(D_2 = E_j\) is an exceptional divisor different from \(E_1\), then \((D_1, D_2) = 0\).  Otherwise, \([D_2] = aH - \sum_{i=1}^r b_i E_i\), with \(b_i \geq 0\), in which case \((D_1 , D_2) = b_i \geq 0\).

For the second part, write \(D_1 = aH - \sum_{i=1}^r b_i E_i\) and \(D_2 = cH - \sum_{i=1}^r d_i E_i\).  That \((D_1, D_2) \geq 0\) yields
\[
(n-1) ac \geq \sum_{i=1}^r b_i d_i.
\]
By Lemma \ref{maininequality}, this means that the codimension-\(2\) cycle \([D_1 \cap D_2] = ac H - \sum_{i=1}^r b_i d_i\) is contained in the span of linear cycles.
\end{proof}

\subsection*{Easy Lemmas} Here we collect a couple of geometric lemmas that we will use repeatedly.

\begin{lemma}\label{lem-intersectnef}
Suppose that \(E\) is an effective divisor and that \(P\) is a nef divisor.  If \(Y\) is an irreducible, effective variety of dimension \(k\) which is not contained in \(E\), then \(P^{k-1} \cdot E \cdot Y \geq 0\).
\end{lemma}
\begin{proof}
The intersection \(E \cdot Y\) is a (possibly empty) cycle of dimension $k-1$ by assumption. Since \(P\) is nef, it follows that \(P^{k-1} \cdot E \cdot Y \geq 0\).
\end{proof}

\begin{lemma}\label{lem-multiplicityline}
Let $Y \subset X^n_\Gamma$ be an irreducible variety of dimension $k$, not contained in any exceptional divisor $E_i$, with class $aH_k - \sum_{i=1}^e b_i E_{i,k}$. If $b_i + b_j > a$ for two indices $i \not= j$, then $Y$ contains the line through $p_i$ and $p_j$ with multiplicity at least $b_i + b_j -a$. 
\end{lemma}

\begin{proof}
The  base locus of the linear system $|H-E_i -E_j|$ is the line $l_{i,j}$ spanned by $p_i$ and $p_j$. Consequently, the intersection $(H - E_i -E_j)^{k-1} \cdot Y$ is an effective $1$-cycle $Z$. Express $$Z = \alpha l_{i,j} + u,$$ where $u$ is a $1$-cycle not containing $l_{i,j}$. Since $$-\alpha  \leq (H - E_i -E_j) \cdot Z =  a - b_i -b_j < 0,$$ we conclude that $\alpha \geq b_i + b_j -a$. Hence, $Y$ must have multiplicity at least $b_i + b_j -a$ at every point of $l_{i,j}$.
\end{proof}

\section{Points in linearly general position}\label{sec-lingen}
In this section, we study $\Effb_k(X_{\Gamma}^n)$ when the cardinality of $\Gamma$ is small and the points of $\Gamma$ are in linearly general  position. Our main theorem is the following.
\begin{theorem}\label{lingen}
Let $\Gamma$ be a set of $r$ points in $\P^n$ in linearly general  position. If $$r \leq \max\left(n+2, n + \frac{n}{k} \right),$$  then $\Effb_k(X_{\Gamma}^n)$ is linearly generated.
\end{theorem}

The proof will be by induction on $k$ and $n$. We first single out the case $k=1$.

\begin{lemma}\label{lem-baseCurve}
Let $\Gamma$ be a set of $r \leq 2n$ points in $\P^n$ in linearly  general position. Then $\Effb_1(X_{\Gamma}^n)$ is linearly generated.
\end{lemma}
\begin{proof}
Let $B$ be an irreducible curve. By Lemma \ref{lem-ab}, we may assume that $B$ is not contained in any of the exceptional divisors and has class $aH_1 - \sum _{i=1}^r b_i E_{i,1}$ with $a \geq b_i \geq 0$. Any $r \leq 2n$ points in  linearly general position are cut out by quadrics \cite[Lecture 1]{Harris}. Consequently, there is a quadric whose proper transform has class $[Q]= 2H - \sum_{i=1}^{r} E_i$ in $X_{\Gamma}^n$ and does not contain $B$. Hence, $B$  has nonnegative intersection with $Q$ and satisfies \(2a \geq \sum_{i=1}^{r} b_i\). By Lemma \ref{maininequality},  the class of $B$ is spanned by the classes of lines.
\end{proof}

Next, we  study the case when $r \leq n+1$. In this case, $X_{\Gamma}^n$ is toric and the effective cones are generated by torus-invariant cycles (see e.g.\ \cite[Prop.\ 3.1]{Li}). For the reader's convenience we will give a simple independent proof. 
\begin{lemma}
\label{toriccase}
Let $\Gamma$ be a set of $r \leq n+1$ linearly general points in $\P^n$. The cone $\Effb_k(X_{\Gamma}^n)$ is linearly generated for any $k$.
\end{lemma}
\begin{proof}
Let $\Gamma' \subset \Gamma$ be two sets with cardinality $r$ and $n+1$, respectively. Then the proper transform of any effective cycle in $X_{\Gamma'}^n$ is an effective cycle in $X_{\Gamma}^n$. Consequently, if $\Effb_k(X_{\Gamma}^n)$ is linearly generated, then $\Effb_k(X_{\Gamma'}^n)$ is also linearly generated. Hence, without loss of generality, we may assume that $r=n+1$. Let $Y$ be an irreducible $k$-dimensional variety in $X_{\Gamma}^n$ with class $$[Y]=aH_k -\sum_{i=1}^{n+1} b_i E_{i,k}.$$  By Lemma \ref{lem-ab}, we may assume that $Y$ is not contained in an exceptional divisor and that $a \geq b_i \geq 0$. We proceed by induction on $k$ and $n$. Up to reordering the points, we may assume $b_1 \geq b_2 \geq \cdots \geq b_{n+1}$. Let $L$ be the proper transform of the $\P^{n-1}$ spanned by the first $n$ points.  First, suppose $Y$ is contained in $L$. Since $L$ is isomorphic to the blow-up of $\P^{n-1}$ in $n$ points, by induction on $n$ with base case Theorem \ref{divisorcase}, we conclude that the class of $Y$ is in the span of linear spaces. Otherwise, $Y \cap L$ is an effective cycle of dimension $k-1$ in $L$. Write $H_{L, k-1}$ and $E_{L, i,k-1}$ for the restriction of $H_k$ to $L$ and the $(k-1)$-dimensional linear space in the exceptional divisor $E_{L,i}$ of the blow-up of $L$ at $p_i$. Then we have $$[Y \cap L] = aH_{L, k-1} -\sum_{i=1}^{n} b_i E_{L, i, k-1}.$$ By induction on $n$ with base case Lemma \ref{lem-baseCurve}, $Y \cap L$ is in the span of linear spaces. In particular, $ka \geq \sum_{i=1}^n b_i.$ Hence, $(k+1)a \geq \sum_{i=1}^{n+1} b_i.$ By Lemma \ref{maininequality}, the class of $Y$ is in the span of linear spaces. This concludes the proof. 
\end{proof}

We can now complete the proof of Theorem \ref{lingen}.
\begin{proof}[Proof of Theorem \ref{lingen}]
We preserve the notation from the proof of Lemma \ref{toriccase} and argue similarly. Suppose that $Y$ is an irreducible $k$-dimensional variety on $X_{\Gamma}^n$ with class
$$[Y] = a H_k - \sum_{i=1}^r b_i E_{i,k}.$$ We may assume that $Y$ is not contained in an exceptional divisor and, by reordering the points, we have that $$a \geq b_1 \geq \cdots \geq b_r \geq 0.$$  Let $L$ be the $\P^{n-1}$ passing through the points $p_1, \ldots, p_n$. If $Y$ is contained in $L$, then its class is linearly generated by Lemma \ref{toriccase}. Otherwise, $Y \cap L$ is an effective cycle of dimension $k-1$ with class
\[
[Y\cap L] = aH_{L,k-1} - \sum_{i=1}^n b_i E_{L,i, k-1}
\]
which is also linearly generated by Lemma \ref{toriccase}. Therefore, the class \([Y \cap L]\) can be written as a combination of linear classes \(H_{L, k-1} - \sum_{\abs{I} = k} E_{L,i, k-1}\) and \(E_{L,i,k-1}\)
\[
[Y \cap L] = \sum_{j=1}^{a} \alpha_j \left(H_{L,k-1} -  \sum_{\abs{I} = k} E_{L,i, k-1}\right) + 
 \sum_{j=1}^{n} \beta_j E_{L, j, k-1}.
\]

Each of the classes in this sum is effective, with those on the left the classes of \(\P^{k-1}\) through \(k\) of the points in \(L\).  By taking cones over these classes, we obtain a \(\P^{k}\) on \(X\), passing through an additional one of the points \(p_i\) with \(i > n\).
Since there are  \(a\) planes available, if \(\sum_{i=n+1}^r b_i \leq a\), the class \(Y\) can be expressed as a sum of linear cycles.

Observe that $$a k \geq \sum_{i=1}^n b_i \geq n b_n, \ \mbox{and so} \ b_j \leq b_n \leq  \frac{a k}{n} \ \mbox{for} \ j \geq n.$$  This implies that if \((r-n)\frac{k}{n} \leq 1\) or equivalently if $r \leq n + \frac{n}{k}$, the classes of all effective cycles are in the span of the classes of linear spaces.

If $k \leq \frac{n}{2}$, then $n+2 \leq n + \frac{n}{k}$ and the theorem is proved. If $k > \frac{n}{2}$, then $n+1 < n + \frac{n}{k} < n+2$ and we need to settle the case  $r=n+2$. There is a rational normal curve through any $n+3$ points in  linearly general  position in $\P^n$ \cite[Lecture 1]{Harris}. Consequently, given an effective divisor $D$, there exists a  rational normal curve $C$ containing the points but not contained in $D$. Hence, $C \cdot D \geq 0$ and all effective divisors satisfy $na \geq \sum_{i=1}^{n+2} b_i$. We recover the linear generation result of Theorem \ref{divisorcase}. By Lemma \ref{lem-baseCurve}, the curve classes are also linearly generated.  By induction assume that for all $m< n$ and all $k < m$, the effective cone of $k$ cycles of the blow-up of $\P^m$ in $m+2$ linearly general points is linearly generated. We carry out the inductive step for $\P^n$. Let $Y, L$ be as above. By Lemma \ref{toriccase}, we may assume that $Y$ is not contained in $L$. If $b_{n+1} + b_{n+2} \leq a$, then we already proved that  the class of $Y$ is linearly generated. If $b_{n+1} + b_{n+2} > a$, then, by Lemma \ref{lem-multiplicityline}, $Y$ contains the line $l_{n+1, n+2}$ spanned by $p_{n+1}$ and $p_{n+2}$ with multiplicity at least $b_{n+1} + b_{n+2} - a$. Let $p_0$ denote the point of intersection $L \cap l_{n+1, n+2}$. Then the proper transform of $L \cap Y$ is an effective cycle in the blow-up of $L$ in $p_0, p_1, \dots, p_n$ with class $$a H_{L, k-1} - (b_{n+1} + b_{n+2} -a) E_{L,0,k-1} - \sum_{i=1}^n b_i E_{L, i, k-1}.$$ By induction on $n$, this class is linearly generated. Hence, $$ka \geq  b_{n+1} + b_{n+2} -a + \sum_{i=1}^n b_i, \quad \mbox{therefore} \quad (k+1) a \geq \sum_{i=1}^{n+2} b_i.$$ By Lemma \ref{maininequality}, the class of $Y$ is linearly generated. This concludes the proof. 
\end{proof}

\begin{example}\label{ex-linearSharp}
Lemma \ref{lem-baseCurve} is sharp in the sense that there exist sets $\Gamma$ of $r > 2n$ points in general linear position such that $\Effb_1(X_{\Gamma}^n)$ is not linearly generated. For example, let $\Gamma$ be  $r > 2n$ points on a rational normal curve $C$ in $\P^n$. Points on a rational normal curve are in general linear position \cite{Harris}. Then the class of the proper transform of $C$ is $$nH_1 - \sum_{i=1}^r E_{i,1}.$$ Since $r>2n$, this class  cannot be in the span of the classes of lines. In the next section, we will see that we can improve the bounds for linear generation exponentially if instead of assuming that $\Gamma$ is linearly general, we assume $\Gamma$ is a set of very general points in $\P^n$. 

More generally, let $Y$ be the cone over a rational normal curve of degree $n-k+1$ with vertex $V$ a $\P^{k-2}$. Let $\Gamma$ be the union of a set of $k-1$ general points $p_1, \dots, p_{k-1}$ in $V$ and a set of $r-k+1$ general points $p_k, \dots, p_r$ on $Y$. Then $\Gamma$ is in general linear position. The class of the proper transform of $Y$ is 
$$(n-k+1)H_k -\sum_{i=1}^{k-1} (n-k+1)E_{i,k} - \sum_{i=k}^r E_{i,k},$$ which cannot be in the span of linear spaces if $r > 2n-k+1$.  Consequently, we conclude the following.
\begin{proposition}\label{prop-lingenSharp}
There exists sets $\Gamma$ of $r>2n-k+1$ points in general linear position in $\P^n$ such that $\Effb_k(X_{\Gamma}^n)$ is not linearly generated. 
\end{proposition}

In view of Proposition \ref{prop-lingenSharp}, it is natural to ask whether the bound in Theorem \ref{lingen} can be improved to $r \leq 2n-k+1$.
\begin{question}\label{ques-linear}
Assume that $\Gamma$ is a set of $r$ linearly general points in $\P^n$ such that $$\max\left(n+2, n + \frac{n}{k}\right) <  r \leq 2n-k+1.$$ Is $\Effb_k(X_{\Gamma}^n)$ linearly generated?
\end{question}
The answer is affirmative for curves and divisors. We will shortly check that for $2$-cycles in $\P^4$ the answer is also affirmative. In Theorem \ref{thm-verygenerallingen} we will see that that the answer is also affirmative if the points are very general. In view of this evidence, we expect the answer to Question \ref{ques-linear} to be affirmative.
\end{example}

\begin{remark}\label{remark-scroll}
The dimension of the space $\mathcal{S}_{n-k, k+1} (\P^n)$ of scrolls of dimension $n-k$ and degree $k+1$ in $\P^n$ is $$2n + 2nk - k^2 -2$$ \cite[Lemma 2.4]{CoskunGW}. There are scrolls in $\mathcal{S}_{n-k, k+1}(\P^n)$  passing through $2n-k+2$ points (see \cite{CoskunScrolls} for the surface case). Hence, the family of scrolls passing through $2n-k+1$ points covers $\P^n$. By Lemma \ref{maininequality}, an affirmative answer to Question \ref{ques-linear} is equivalent to the statement that every effective $k$-dimensional cycle intersects the proper transform of a scroll passing through the $2n-k+1$ points non-negatively. 
\end{remark}

\begin{question}\label{ques-scroll}
Let $\Gamma$ be $2n-k+1$ linearly general points in $\P^n$. For every effective $k$-cycle $Y$ in $X_{\Gamma}^n$, does there exist a scroll  $S$ of dimension $n-k$ and degree $k+1$ such that the proper transform $S$ in $X_{\Gamma}^n$ intersects $Y$ in finitely many points?
\end{question}
By Remark \ref{remark-scroll}, an affirmative answer to Question \ref{ques-scroll} implies an affirmative answer to Question \ref{ques-linear}. 

\subsection*{Effective $2$-cycles on the blow-up of $\P^4$ at $7$ points} We now verify that the answer to Question \ref{ques-linear} is affirmative for two-cycles in $\P^4$. The argument is subtle because we need to verify linear generation for {\em every} configuration of $7$ points in linear general position, rather than just very general configurations of points.

\begin{theorem}\label{thm-7inp4}
Let $\Gamma$ be $7$ linearly general points on $\P^4$. Then the cone $\Effb_2(X_{\Gamma}^4)$  is linearly generated.
\end{theorem}
\begin{proof}
There is a unique rational normal quartic curve $R$ containing $7$ linearly general points in $\P^4$ \cite{Harris}. The secant variety $\Sec(R)$ to $R$ is a cubic hypersurface which is double along $R$. Hence, its proper transform $\overline{\Sec}(R)$ in $X_{\Gamma}^4$ has class $3H - \sum_{i=1}^7 2 E_i$. 

Let $Y$ be an irreducible surface in $X_{\Gamma}^4$. Without loss of generality, we may assume that $Y$ is not contained in an exceptional divisor and has class $aH_2 - \sum_{i=1}^7 b_i E_{i,2}$ with $a \geq b_i \geq 0$.
First, suppose that $Y$ is not contained in $\overline{\Sec}(R)$. The class of a quadric \([Q] = H - \sum_{i=1}^7 E_{i}\) is nef, and so by Lemma~\ref{lem-intersectnef} we have $$Y \cdot Q \cdot \overline{\Sec}(R) = 6a - \sum_{i=1}^7 2b_i \geq 0.$$   Lemma \ref{maininequality} implies that $[Y]$ is in the span of the classes of planes.

We are reduced to showing that if $Y \subset \overline{\Sec}(R)$, then $[Y]$ is in the span of the classes of planes. Let $\mathcal{S}_3$ denote the space of cubic surface scrolls containing the points of $\Gamma$.  We will show the following.
\begin{theorem}\label{thm-scroll}
The proper transform $\overline{S}$ of a general member $S \in \mathcal{S}_3$ intersects $\overline{\Sec}(R)$ in an irreducible curve $B$ whose projection to $\Sec(R)$ is a degree $9$ curve double at the points of $\Gamma$. Furthermore, the curve $B$ can be made to pass through a general point of $\overline{\Sec}(R)$.
\end{theorem} 

Assume Theorem \ref{thm-scroll}.  Let $p \in \overline{\Sec}(R)$ be a general point not contained in $Y$. Hence, an irreducible curve $B$ passing through $p$ intersects $Y$ in finitely many points. Let $\overline{S}$ be the proper transform of a scroll $S \in \mathcal{S}_3$ containing $p$ and intersecting $\overline{\Sec}(R)$ in an irreducible curve. We conclude that $\overline{S}$ and $Y$ intersect in finitely many points, hence their intersection number is positive. Therefore, $$[\overline{S}] \cdot [Y] = (3 H_2 - \sum_{i=1}^7 E_{i,2}) \cdot (aH_2 - \sum_{i=1}^7 b_i E_{i,2}) = 3a - \sum_{i=1}^7 b_i \geq 0.$$ By Lemma \ref{maininequality}, we conclude that $[Y]$ is in the span of the classes of planes.

There remains to prove Theorem \ref{thm-scroll}, which we will do via a series of claims. We first set some notation.
\begin{notation}
Let $l_{i,j}$ denote the line spanned by $p_i, p_j \in \Gamma$ and let $\Pi_{i_1, \dots, i_l}$ denote the linear space spanned by $p_{i_1},\dots, p_{i_l} \in \Gamma$. Let $\Gamma_{i_1, \dots, i_l}$ denote the set of points $p_{i_1}, \dots, p_{i_l}$. Let $l$ be  the line of intersection $\Pi_{1,2,3,4} \cap \Pi_{5,6,7}$ and, for $5 \leq i < j \leq 7$, let $z_{i,j}$ denote the point of intersection $\Pi_{1,2,3,4} \cap l_{i,j}$. Since the points are in linearly general position, the line $l$  does not intersect the lines $l_{i,j}$ for $1 \leq i <j \leq 4$ and intersects the planes $\Pi_{i,j,k}$, for $1\leq i<j<k \leq 4$, in a unique point different from $z_{i,j}$. 
\end{notation}
 
Next, we recall a compactification $\overline{\mathcal{S}}_3$ of $\mathcal{S}_3$. Every irreducible cubic scroll induces a degree $3$ rational curve in the Grassmannian $\mathbb{G}(1,4)$ of lines in $\P^4$. We can compactify the space of degree $3$ rational curves in $\mathbb{G}(1,4)$ via the Kontsevich moduli space. Hence, we can take the closure of $\mathcal{S}_3$ in the Kontsevich moduli space (see \cite[\S 3]{CoskunScrolls} for details). More precisely, let $\overline{\mathcal{M}}_{0,7}(\mathbb{G}(1,4),3)$ denote the Kontsevich moduli space of $7$-pointed genus-$0$ maps of degree $3$ to $\mathbb{G}(1,4)$.  It is equipped with $7$ evaluation morphisms $\ev_i: \overline{\mathcal{M}}_{0,7}(\mathbb{G}(1,4),3) \rightarrow \mathbb{G}(1,4)$, $1 \leq i \leq 7$. Define $$\overline{\mathcal{S}}_3 = \bigcap_{i=1}^7 \ev_i^{-1} (\Sigma_3(p_i)),$$ where $\Sigma_3(p_i)$ denotes the Schubert variety of lines containing $p_i$.

\begin{claim}\label{claim-irreducible}
The space $\overline{\mathcal{S}}_3$ is irreducible of dimension $4$.
\end{claim}
\begin{proof}
The locus $\overline{\mathcal{T}}= \cap_{i=1}^4 \ev_i^{-1}(p_i)$ in the Kontsevich moduli space $\overline{\mathcal{M}}_{0,4}(\P^3, 3)$ of $4$-pointed genus $0$ maps of degree $3$ to $\P^3$ provides a compactification of the space of twisted cubic curves in $\Pi_{1,2,3,4}$ containing $\Gamma_{1,2,3,4}$. Since $\overline{\mathcal{M}}_{0,4}(\P^3, 3)$ is irreducible of dimension $16$, every component of $\overline{\mathcal{T}}$ has dimension at least $4$. 

If a twisted cubic $T$ is irreducible, then any finite set of points on $T$ is linearly general. Furthermore, given $6$ linearly general points in $\P^3$, there is a unique twisted cubic curve containing them. Consider the incidence correspondence $I= \{(T, q_1, q_2)| q_1, q_2 \in T\}$, where $T$ is a twisted cubic curve containing the set of points $\Gamma_{1,2,3,4}$ and $q_1, q_2$ are points such that $\Gamma_{1,2,3,4} \cup \{ q_1, q_2\}$ are in linearly general position. The incidence correspondence $I$ is irreducible of dimension $6$ since it is isomorphic to an open  subset of $\P^3 \times \P^3$. It dominates the space of twisted cubic curves containing $p_1, \dots, p_4$ via the first projection. Since the fibers of the first projection are two-dimensional, we conclude that the space of irreducible twisted cubics containing $\Gamma_{1,2,3,4}$ is irreducible of dimension $4$. 

Since there are no connected curves of degree two or one containing $4$ points in linearly general position in $\P^3$,  any map in $\overline{\mathcal{T}}$ is birational to its image. If there is a reducible curve of degree $3$ containing $\Gamma_{1,2,3,4}$, either a degree two curve must contain $3$ of the points or a line must contain two of the points. In either case, it is easy to see that there is a $3$-dimensional family of reducible cubics containing $\Gamma_{1,2,3,4}$. Hence, these cannot form a component of $\overline{\mathcal{T}}$ and $\overline{\mathcal{T}}$ is irreducible. 

Furthermore, $2$ additional points $q_1, q_2$ impose independent conditions on twisted cubics unless they are coplanar with three of the points in $\Gamma_{1,2,3,4}$  or one of the points is collinear with two of the points in $\Gamma_{1,2,3,4}$. If $q_1$ is collinear with $p_1, p_2$, then there is a $1$-parameter family of reducible cubics containing the line $l_{1,2}$. Similarly, if $q_1$ and $q_2$ are in $\Pi_{1,2,3}$ but no $4$ of the points are collinear, then there is a $1$-parameter family of reducible cubics containing the conic through $\Gamma_{1,2,3} \cup \{ q_1, q_2\}$. If $q_1, q_2$ are collinear with $p_1$ and $p_2$, there is a three-parameter family of reducible cubics containing $l_{1,2}$.  Recall that  $l= \Pi_{1,2,3,4} \cap \Pi_{5,6,7}$.  In particular, the subset of $\overline{\mathcal{T}}$ that parameterizes twisted cubics  incident (respectively, secant) to $l$ has dimension $3$ (respectively, $2$) since any pair of distinct points impose independent conditions on twisted cubics. Similarly,  the locus of twisted cubics in $\overline{\mathcal{T}}$ passing through $z_{5,6}$ has dimension $2$.

Since $\overline{\mathcal{M}}_{0,7}(\mathbb{G}(1,4),3)$ is irreducible of dimension $25$ \cite[\S 2]{CoskunScrolls}, every irreducible component of $\overline{\mathcal{S}}$ has dimension at least $4$. Let $T$ be a twisted cubic curve containing $\Gamma_{1,2,3,4}$ and not secant to the line $l$ and not containing the points $z_{5,6}$, $z_{5,7}$ and $z_{6,7}$. Then there is a unique cubic scroll $S$ containing  $T$ and passing through $p_5, p_6, p_7$ \cite[Example A1]{CoskunScrolls}. Briefly, take a general $\P^3$ containing $\Pi_{5,6,7}$. This $\P^3$ intersects $T$ in $3$ points $r_1, r_2, r_3$. There is a unique twisted cubic curve $T'$ containing $r_1, r_2, r_3$ and $\Gamma_{5,6,7}$. The curves $T$ and $T'$ are both isomorphic to $\P^1$ and there is a unique isomorphism $\phi$ taking $r_i \in T$ to $r_i \in T'$. Then the surface $S_{T, T'}$ swept out by lines joining the points that correspond under $\phi$ is the unique cubic scroll containing $T$ and $\Gamma_{5,6,7}$. If $T$ contains the point $z_{5,6}$ or is secant to the line $l$, then there is a $1$-parameter family of choices for $T'$. Once we fix $T$ and $T'$, the scroll is uniquely determined by a similar construction. Since the locus of $T$ containing $z_{5,6}$ or secant to $l$ has codimension $2$, this locus cannot form a component of $\overline{\mathcal{S}}_3$. Finally, reducible cubic surfaces containing $\Gamma$ must contain a plane through $3$ of the points and a quadric surface through the remaining $4$ points. There is a $2$-dimensional family of such surfaces and they do not give rise to a component in $\overline{\mathcal{S}}_3$ (see \cite{CoskunScrolls}). We conclude that $\overline{\mathcal{S}}_3$ is irreducible of dimension $4$.
\end{proof}

\begin{claim}\label{claim-notcontained}
There exists a dense open set $U \subset \overline{\mathcal{S}}_3$ such that $S \not\subset \Sec(R)$ for   $S \in U$. Furthermore, we may assume that $S$ contains a general point of $\Sec(R)$.
\end{claim}
\begin{proof}
It suffices to exhibit one $S \in \overline{\mathcal{S}}_3$ such that $S \not\subset \Sec(R)$. Given $7$ points in general linear position and $2$ general additional points, \cite[Example A1]{CoskunScrolls} shows that there are $2$ cubic scrolls containing these nine points. In particular, if we take one of the two additional points outside $\Sec(R)$, we obtain a scroll not contained in $\Sec(R)$. Furthermore, a general twisted cubic in $\Pi_{1,2,3,4}$ containing $\Gamma_{1,2,3,4}$ intersects $\Sec(R)$ in a another point $q$. Consequently, the construction in the proof of Claim \ref{claim-irreducible} exhibits a cubic scroll containing $q$ and not contained in $\Sec(R)$. Since the space $\overline{\mathcal{S}}_3$ is irreducible, the general scroll containing a general point of $\Sec(R)$ and $\Gamma$ will not be contained in $\Sec(R)$.
\end{proof}

\begin{claim}\label{claim-curves}
There exists a dense open set  $U \subset \overline{\mathcal{S}}_3$ such that for $S \in U$ the following hold:
\begin{enumerate}
\item The intersection $\overline{S} \cap \overline{\Sec}(R) \cap E_i$ is a finite set of points in $X_{\Gamma}^4$ for every $1 \leq i \leq 7$.
\item The scroll $S$ does not contain any lines $l_{i,j}$ for $1 \leq i < j \leq 7$.
\item The scroll $S$ does not contain any conics through $3$ of the points in $\Gamma$.
\item The scroll $S$ does not contain a twisted cubic curve through $5$ of the points of $\Gamma$.
\item The scroll $S$ does not contain the rational normal quartic $R$.
\item The scroll $S$ does not contain a quintic curve double at one of the points of $\Gamma$ and passing through the others.
\item The directrix of the scroll does not contain any of the points in $\Gamma$.
\end{enumerate}
\end{claim}
\begin{proof}
Since each of these conditions are closed conditions and $\overline{\mathcal{S}}_3$ is irreducible, it suffices to exhibit one element $S \in \overline{\mathcal{S}}_3$ satisfying each condition. For (1), there exists a twisted cubic containing $\Gamma_{1,2,3,4}$ with any tangent line at $p_1$ (for example, the reducible twisted cubic consisting of any line through $p_1$ and a conic through $\Gamma_{2,3,4}$). Hence, the tangent spaces to the scrolls at $p_1$ sweep out $E_1$ and there exists $S$ such that $\overline{S} \cap E_1\not\subset \overline{\Sec}(R)$. By permuting indices, we conclude (1). 

For (2) and (3),  take the scroll $S_{T, T'}$ constructed in the proof of Claim \ref{claim-irreducible}. Since  $\Pi_{1,2,3,4} \cap S_{T, T'} = T$, this scroll does not contain any of the linear $l_{i,j}$ with $1 \leq i < j \leq 4$ or any conic passing through any of the three points in $\Gamma_{1,2,3,4}$. By permuting indices, we conclude (2) and (3). 

Since a twisted cubic curve spans a $\P^3$ and the points are in linearly general position (4) is clear. For (5), (6) and (7), it is more convenient to exhibit a reducible scroll satisfying these properties. Let $S$ be the union of the plane $\Pi_{5,6,7}$ and a general quadric surface $Q$ containing $l$ and $\Gamma_{1,2,3,4}$. After choosing a point of $l$, this surface determines a point $p$ of $\overline{\mathcal{S}}_3$ \cite{CoskunScrolls}. The directrix line is then the unique line on the quadric $Q$ intersecting $l$ at $p$. Hence, (7) holds. Since $R$ is irreducible and nondegenerate, it cannot be contained in this surface. Suppose there is a quintic curve $F$ in $S$ containing $\Gamma$ and double at $p_1$. Since $p_5, p_6, p_7$ are not collinear, $F$ must intersect $\Pi_{5,6,7}$ in a curve of degree at least $2$. Hence, $F$ intersects $Q$ in a curve of degree $3$ containing $\Gamma_{2,3,4}$ and double at $p_1$. Any cubic double at $p_1$ must contain the line of ruling through $p_1$. Since $\Gamma_{1,2,3,4}$ are linearly general there cannot be a degree 2 curve through these points on $Q$. After permuting indices, we conclude (6) holds.
\end{proof}

\begin{claim}\label{claim-irreduciblecurve}
There exists a dense open set $U \subset \overline{\mathcal{S}}_3$ such that for $S \in U$ the intersection $S \cap \Sec(R)$ is an irreducible degree $9$ curve double along $\Gamma$. 
\end{claim}
\begin{proof}
By Claim \ref{claim-notcontained} and Claim \ref{claim-curves}, we can find a scroll $S \not\subset \Sec(R)$ and satisfying the conclusions of Claim \ref{claim-curves}. The intersection $S \cap \Sec(R)$ is a curve $B$ of degree $9$  double along $\Gamma$. We need to show that $B$ is irreducible. Recall that a smooth cubic scroll is isomorphic to the blow-up of $\P^2$ at a point. Its Picard group is generated by the directrix $e$  (the curve of self-intersection $-1$) and the class of a fiber line $f$. The intersection numbers are $$e^2=-1, \quad e\cdot f=1,  \quad f^2=0.$$ The effective cone is spanned by $e$ and $f$. The canonical class is $-2e-3f$ and the class of $B$ is $3e+6f$. The degree of a curve $ke+mf$ is $k+m$. If $k>m$, then any representative contains $e$ with multiplicity $k-m$. By adjunction, the arithmetic genus of a curve in the classes $e+ mf$, $2e+mf$ and $3e+mf$ are $0$, $m-2$ and $2m-5$, respectively. 

It is now straightforward, but somewhat tedious to check that $B$ cannot be reducible. Suppose $B$ is reducible. Unfortunately, $B$ can have many components. Write $B= B_1 \cup B_2$, where the class of $B_1$ is $ke + mf$ with $2 \leq k \leq 3$ and assume that $B_1$ does not contain any fibers as components. Furthermore, if $k=3$, we may assume that $B_1$ is irreducible. Otherwise, we can regroup a component with class $e+m'f$ with $B_2$.  Then the class of $B_2$ is $(3-k)e + (6-m)f$ and every fiber component of $B$ is included in $B_2$. By Claim \ref{claim-curves} (2), a curve with class $mf$ can be double at most in $0 \leq d \leq \frac{m}{2}$ points of $\Gamma$ in which case it can contain at most $m-2d$ of the remaining points of $\Gamma$. We tabulate the possibilities for curves with class $e+mf$.
\bigskip

\begin{tabular}{| l | l | l | l |} 
\hline 
class & $\#$ double points of $\Gamma$ & $\#$ remaining points of $\Gamma$ resp. & Reason \\
\hline 
$e$ & $0$ & $0$ & Claim \ref{claim-curves} (7)  \\
\hline
$e+f$ & $0$ & $2$ & Claim \ref{claim-curves} (3,2) \\
\hline
$e+2f$ & $0$ or $1$ & $4$ or $1$ & Claim \ref{claim-curves} (4,3,2) \\
\hline
$e+3f$ & 0, 1 or 2& 6, 3 or 0 & Claim \ref{claim-curves} (7,5,4,3,2) \\
\hline
$e+4f$ & 0, 1 or 2 & 7, 5 or  2 & Claim \ref{claim-curves} (7,5,4,3,2) \\
\hline
\end{tabular}
\bigskip

First, suppose $B_1$ has class $3e+mf$. By assumption, it is irreducible and by arithmetic genus considerations can have at most $2m-5$ nodes. On the other hand, $B_2$ can pass through at most $6-m$ of the points. We have $2m-5 + 6 -m = m+1 < 7$ if $m < 6$. Hence, such a curve cannot be double at all the points of $\Gamma$.

We may therefore assume that the class of $B_1$ is $2e + mf$ and the class of $B_2$ is $e + (6-m)f$. If $B_1$ is reducible, then it can have at most $2$ components with classes $e+m_1f$ and $e+m_2 f$. An inspection of the above table shows that it is not possible to make $B$ double at all points of $\Gamma$. If $B_1$ is irreducible, then $m\geq 2$ and its arithmetic genus is $m-2$. Hence, the maximal number of double points on $B_1$ is $m-2$. If $m=2$, $B_1$ can contain at most $6$ of the points of $\Gamma$ by Claim \ref{claim-curves} (5) and it is smooth at those points. Hence, $B$ cannot be made double at all points of $\Gamma$ by the last line of the table. If $m=3$ and $B_1$ has a double point, then by Claim \ref{claim-curves} (6) $B_1$ contains at most $5$ other points of $\Gamma$. By the second to last row of the table, $B$ cannot be double at all points of $\Gamma$. If $m \geq 4$, an easy inspection of the first three rows of the table show that $B$ can have at most $6$ double points. We conclude that $B$ is irreducible.
\end{proof}
This concludes the proof of Theorem \ref{thm-scroll} and consequently of Theorem \ref{thm-7inp4}.

\end{proof}

\section{Non-linearly generated cones}\label{sec-nonlingen}
Recall that $X_r^n$ denotes the blow-up of $\P^n$ in $r$ very general points. In this section, we study the cones of effective cycles on $X_{r}^n$. Our first result completely characterizes when the cone of curves is linearly generated.

\begin{proposition}\label{prop-verygeneralCurve}
The cone $\Effb_1(X_r^n)$ is linearly generated if and only if $r \leq 2^n$.
\end{proposition}

\begin{proof}
We first observe  that the linear system of quadrics through $2^n$ very general points is nef.    Choose $n$ general quadrics $Q_1, \dots, Q_n$ in $\P^n$.  By Bertini's theorem, the intersection of these quadrics is a set of $2^n$ points in $\P^n$.  Let $X_0$ be the blow-up of $\P^n$ at these points.  We claim that $D= 2H - \sum_{i=1}^{2^n} E_i$ is nef on $X_0$. Note that the proper transforms of $Q_i$ have class $D$ and $D$ has positive degree on curves contained in exceptional divisors $E_i$. Since the intersection $Q_1 \cap \cdots \cap Q_n$ is finite, if $B$ is a curve on $X_0$ not contained in an exceptional divisor, there is a quadric $Q_i$ whose proper transform does not contain $B$. Consequently,  $D \cdot B \geq 0$ and $D$ is nef.
 By \cite[Prop. 1.4.14]{Lazarsfeld}, \(2H - \sum_{i=1}^{2^n} E_i\) is nef for very general configurations of \(2^n\) points as well. We conclude that if $r \leq 2^n$, an effective curve class in $X_{r}^n$ satisfies the inequalities in the assumptions of Lemma~\ref{maininequality}, and so every curve class is a linear combination of classes of lines.

The top self-intersection of the class $Q= 2H -\sum_{i=1}^{r} E_i$  on $X_{r}^n$ is given by $2^n - r$. Hence, if  $r> 2^n$, then the top self-intersection of $Q$ is negative and $Q$ cannot be nef by Kleiman's Theorem \cite[Theorem 1.4.9]{Lazarsfeld}.  Suppose the class of every effective curve is in the span of the classes of lines. The cone generated by the classes of lines is a closed cone. Hence, the effective and the pseudoeffective cones coincide. Since every line has nonnegative intersection with $Q$, we conclude that $Q$ is nef. This contradiction shows that there must exist effective curves whose classes are not spanned by the classes of lines.
\end{proof}

We obtain the following consequence.

\begin{corollary}\label{cor-linear}
If $r \geq 2^{n-k+1} +k$, then $\Effb_k(X_r^n)$ is not linearly generated.
\end{corollary}
\begin{proof}
Let $\Gamma$ be a set of $r$ very general points. Project $\Gamma$ from the first $k-1$ points $p_1, \dots, p_{k-1}$ and let $\Gamma'$ be the set of points in $\P^{n-k+1}$ consisting of the images of the remaining points. Then  $\Gamma'$ is a set of $r-k+1$ very general points in $\P^{n-k+1}$.  If $r-k+1> 2^{n-k+1}$, the cone $\Effb_1(X_{r-k+1}^{n-k+1})$ is not linearly generated by Proposition \ref{prop-verygeneralCurve}.  Fix a $1$-cycle $B$ with class $aH_1 - \sum_{i=k}^r b_i E_{i,1}$ that is not in the span of linear spaces. In particular, $2a < \sum_{i=k}^r b_i$. Then the class  $$aH_k - \sum_{i=1}^{k-1} a E_{i,k} - \sum_{i=k}^r b_i E_{i,k}$$ is represented  in $X_r^n$ by the proper transform of the cone over $B$ with vertex the span of $p_1, \dots, p_{k-1}$.  The resulting $k$-cycle is not in the span of $k$-dimensional linear spaces since $(k+1) a < (k-1) a + \sum_{i=k}^r b_i$.
\end{proof}

\begin{question}\label{ques-verygeneral}
If $r < 2^{n-k+1} + k$, is $\Effb_k(X_{r}^n)$ linearly generated?
\end{question}

\begin{remark}
The answer to Question \ref{ques-verygeneral} is affirmative for curves and divisors. For  cycles of intermediate dimension, we do not know any examples with $r= 2^{n-k+1} + k -1$ where the cone is linearly generated. 

There has been a great deal of interest in the construction of cycles that are nef but not pseudoeffective.  Such cycles were constructed on abelian varieties in \cite{DELV}, and on hyperk\"ahler varieties on \cite{Ottem2}.  If Question~\ref{ques-verygeneral} has an affirmative answer, this would give many examples of nef classes that are not pseudoeffective. For example, if $\Effb_3(X_r^6)$ is linearly generated for $16< r < 19$, then the class $4H_3 - \sum_{i=1}^r E_{i,3}$ would be nef but not pseudoeffective; indeed, the self-intersection of this class is negative.
\end{remark}

We can, however, give a linear bound.

\begin{theorem}\label{thm-verygenerallingen}
The cone $\Effb_k(X_r^n)$ is linearly generated if $r \leq 2n-k+1$.
\end{theorem}
\begin{proof}
The theorem is true for $k=1$ by Lemma \ref{prop-verygeneralCurve} and for divisors by Theorem \ref{lingen}.  We will prove the general case by induction on $n$. Assume that the theorem is true for $\Effb_k(X_r^m)$ for $r \leq 2m-k+1$ and all $k < m < n$. 
Let $\Gamma$ be a set of $r$ points such that $\Gamma$ consists of $r-2$ very general points $p_1, \dots, p_{r-2}$ in a hyperplane $L= \P^{n-1}$ and two very general points $p_{r-1}, p_r$ not contained in $L$. Let $L'$ denote the proper transform of $L$ in $X_{\Gamma}^n$. Note that $L' \cong X_{r-2}^{n-1}$. Let $Y$ be an irreducible $k$-dimensional subvariety  of $X_{\Gamma}^n$ not contained in an exceptional divisor with class $a H_k + \sum_{i=1}^r b_i E_i$ on $X_{\Gamma}^n$. If $Y$ is contained in $L'$, then $Y \subset X_{r-2}^{n-1}$. Since $r-2 \leq 2(n-1) - k +1$, by the induction hypothesis the class of $Y$ is linearly generated and $(k+1) a \geq \sum_{i=1}^r b_i$. If $Y$ is not contained in $L'$, then $Z= Y \cap L'$ is an effective cycle of dimension $k-1$. Let $p_0$ denote the intersection of the line $l_{1,2}$ spanned by $p_{r-1}$ and $p_r$ with $L$. Let $\beta = \max(0, b_{r-1}+b_r -a)$. Consider the blow-up $X_{r-1}^{n-1}$ of $L$ along $p_0, p_1, \dots, r_{r-2}$. Then the proper transform of $Z$ is an effective cycle in $X_{r-1}^{n-1}$ with class $$aH_{k-1} - \beta E_{0, k-1} - \sum_{i=1}^{r-2} b_i E_{i, k-1}.$$ Since $r \leq 2n-k+1$, the inductive hypothesis  $r-1 \leq 2(n-1) - (k-1) +1$ is satisfied. We conclude that this class is linearly generated. Consequently, $$ka \geq \beta + \sum_{i=1}^{r-2} b_i \quad \mbox{and hence} \quad (k+1) a \geq \sum_{i=1}^r b_i.$$ By Lemma \ref{maininequality}, the class of $Y$ is linearly generated. By Corollary \ref{cor-semicont}, $\Effb_k(X_r^n) \subset \Effb_k(X_{\Gamma}^n)$ and $\Effb_k(X_r^n)$ is linearly generated. This concludes the induction and the proof of the theorem. 
\end{proof}

\subsection*{The cone $\Effb_2(X_8^4)$} In this subsection, we prove that $\Effb_2(X_8^4)$ is linearly generated.

\begin{theorem}
\label{thm:8ptsp4}
The cone $\Effb_2(X_8^4)$ is linearly generated.
\end{theorem}

To illustrate the range of applicable techniques, we give two different degeneration arguments to prove this result.

\begin{proof}[Proof 1]
Let $\Gamma$ be a configuration of $8$ points in $\P^4$ such that $p_1, \dots, p_7$ are very general points and $p_8$ is a general point on $\Sec(R)$, where $\Sec(R)$ is the secant variety of the rational normal curve $R$ through the points $p_1, \dots, p_7$. Let $Y$ be an irreducible surface in $X_{\Gamma}$. If $Y$ is not contained in $\overline{\Sec}(R)$, then $Y \cap \overline{\Sec}(R)$ is a curve with class $3aH_1 - \sum_{i=1}^7 2b_i E_{i,1} - b_8 E_{8,1}$. The linear system $Q= 2H- \sum_{i=1}^7 E_i - 2E_8$ has base locus the lines $l_{i,8}$ joining $p_i$ to $p_8$ for $1 \leq i \leq 7$. Any effective member of this linear system is a quadric cone with vertex at $p_8$. Let $q_1, \dots, q_7$ denote the projection of $p_1, \dots, p_7$ through $p_8$. These are very general $7$ points in $\P^3$. Hence, the base locus of the linear system of quadric surfaces containing them is the $7$ points. Taking cones with vertex at $p_8$, we conclude that the base locus of the linear system $|Q|$ is the locus of lines $l_{i,8}$, $1\leq i \leq 7$ as claimed. 

None of these lines $l_{i,8}$ are contained in $\Sec(R)$. To see this, take the intersection of $\Sec(R)$ with a hyperplane containing 3 of the points $p_1, \dots, p_7$. The intersection is a cubic surface with finitely many lines. Hence, the general point $q$ does not contain any lines meeting $p_1, \dots, p_3$. Since this is an open condition, the general point of $\Sec(R)$ does not have any lines passing through $p_1, \dots, p_7$. Since $p_8$ was chosen as a general point of $\Sec(R)$, the claim follows.  We conclude that the class $Q$ restricts to a semi-ample, in particular, nef class on $\Sec(R)$. Hence, $Q \Sec(R) Y = 6a - 2\sum_{i=1}^8 b_i \geq 0$ and $Y$ is linearly generated. We may therefore assume that $Y \subset \Sec(R)$. 

Since $p_8$ is a general point of $\Sec(R)$, by Claims \ref{claim-notcontained}, \ref{claim-curves} and \ref{claim-irreduciblecurve}, the space of cubic scrolls containing $p_1, \dots, p_8$ is two-dimensional and there is a nonempty open subset of scrolls such that the proper transform of a scroll $\overline{S}$ intersects $\overline{\Sec}(R)$ in an irreducible curve $B$. An incidence correspondence argument shows that the space of scrolls containing the $8$ points is irreducible of dimension $2$. Consider triples $(Q_1, \Pi, P)$ consisting a quadric cone $Q_1$ with vertex at $p_8$ and containing the other $7$ points, a plane $\Pi$ on $Q_1$ not containing the points $p_1, \dots, p_7$ and a pencil $P$ of quadrics containing $\Pi$ and the $8$ points. This space is irreducible of dimension $3$ and dominates the space of scrolls with one dimensional fibers. Now taking another general point $q$ of $\Sec(R)$, there exists two scrolls containing the $9$ points. We conclude that the curves $B$ are moving curves on $\overline{\Sec}(R)$. Hence, $\overline{S}$ and $Y$ intersect in finitely many points and $3a - \sum_{i=1}^8 b_i \geq 0$. We conclude that $Y$ is linearly generated.
\end{proof}

\begin{proof}[Proof 2]
It suffices to show that the class $\gamma=3H^2-E_{1,2}-\ldots-E_{8,2}$ is nef. In doing so, we may suppose that the points $p_1,\ldots,p_8$ are in special position; a surface $S$ so that $S\cdot \gamma<0$ on a general blow-up specializes to an effective 2-cycle with the same property.

Let $P_1, P_2$ denote the $x_2x_3x_4$-plane and $x_0x_1x_2$-plane in $\P^4$ respectively, and for $i=1,2$ let $q_i$ be a cone over a smooth conic in $P_i$ passing through the three coordinate points. We specialize the eight points $p_1,\ldots,p_8$ so that $p_1, \ldots, p_5$ are the coordinate points in $\P^4$ and $p_6,p_7,p_8$ are general points on the intersection $q_1\cap q_2$. 

On $X$, the strict transform $Q_1$ of $q_1$ has class $2H-\sum E_i-E_1-E_2$. Consider the divisor $D_1=3H-2\sum E_i+E_1+E_2$ and note that $[Q_1]\cdot [D_1]=2\gamma$. We compute that the linear system $|D_1|$ is 2-dimensional. Moreover, the base locus of $D_1$ is one-dimensional and has 18 components; 15 lines and three quartic normal curves. One checks that none of these curves lie on $Q_1$. Indeed, this is easy to check for one particular configuration (and thus it follows a general 8-tuple as above). In particular, ${D_1}|_{Q_1}$ has only finitely many base-points, hence is semiample on $Q_1$. 

Suppose that $S\subset X$ is an irreducible surface. Then if $S$ is not contained in $Q_1$, the intersection $i^*S$ is represented by an effective 1-cycle on $Q_1$ (here $i:Q_1\to X$ is the inclusion). In this case we have by nefness of $D_1$, $2\gamma \cdot S={D_1}|_ {Q_1} \cdot i^*S\ge 0,$ as desired. 

We are therefore led to consider the case $S\subset Q_1$. Now, considering instead the classes $Q_2=2H-\sum E_i-E_3-E_4$ and $D_2=3H-2\sum E_i+E_3+E_4$, and arguing as above, we obtain that a surface $S$ with $S\cdot \gamma<0$ must be contained in $Q_2$, and hence in the intersection $Q_1\cap Q_2$. However, $Q_1\cap Q_2$ is an irreducible surface, so $S=Q_1\cap Q_2$. Finally, note that $[Q_1][Q_2]$ equals $4H^2-2E_{1,2}-2E_{2,2}-2E_{3,2}-2E_{4,2}-E_{5,2}-E_{6,2}-E_{7,2}-E_{8,2}$, which has intersection number 0 with $\gamma$. This completes the proof.
\end{proof}

\bigskip

We immediately deduce the following corollary. 

\begin{corollary}\label{cor-MDS}
If $X_r^n$ is a Mori Dream Space, then $\Effb_k(X_r^n)$ is finitely generated.\end{corollary}

\begin{remark}
Combining Prop.~\ref{thm:8ptsp4} with the degeneration argument of Theorem \ref{thm-verygenerallingen}, it follows that \(\Effb_2(X_r^n)\) is linearly generated for \(r \leq 2n\) as long as \(n \geq 4\).
\end{remark}

We will see in the next section that \(\Effb_2(X_{10}^4)\) is not finitely generated, assuming the SHGH conjecture holds for blow-ups of \(\P^2\) at \(10\) points.  The only remaining case in dimension \(4\) is:

\begin{question}
Is the cone \(\Effb_2(X_9^4)\) linearly generated?
\end{question}

\bigskip

It is not easy to find explicit curves in $X_r^n$ which are not in the span of lines. The following example gives a construction in the case of $9$ very general points in $\P^3$.

\begin{example}
By \cite{CilibertoMiranda}, the class $C_{CM} = 57 H_1 - \sum_{i=1}^{10} 18 E_{i,1}$ on $X_{10}^2$ is represented by a unique irreducible plane curve of genus \(10\).   On \(X_{9}^3\), there is a unique divisor \(Q\) in the class \(2 H_1 - \sum_{i=1}^9 E_{i,1}\), given by the strict transform of the unique quadric through the \(9\) points.  There is a morphism \(i : X_{10}^2 \to X_{9}^3\) identifying the proper transform of \(Q\) with the blow-up of \(\P^2\) at \(10\) points.

A quick calculation shows that the pushforward of the class of $C_{CM}$ to $X_{9}^3$  is $$i_\ast(C_{CM}) = 78 H_1- 21 E_{1,1} - \sum_{i=2}^9 18 E_{i,1}.$$ We have \(21 + 8(18) = 165\), while \(2 \cdot 78 = 156\).  Hence,  this curve is not in the span of the lines.  It does not, however, define an extremal ray on \(\Effb_1(X_{9}^3)\).  In the next  section, we will use a similar construction to show that assuming the SHGH conjecture, the cone $\Effb_1(X_{9}^3)$ is not finitely generated. 

By repeatedly  taking cones over $i_\ast(C_{CM})$, we obtain explicit non-linearly generated codimension-two cycles on \(X^n_{n+6}\) for every $n \geq 3$. 
\end{example}

Complete intersections also provide examples of nonlinearly generated pseudoeffective curve classes, provided that the number of points is large.
\begin{example}
Assume that $d^n \geq r > 2 d^{n-1}$ for some integer $d>2$. Then the divisor class $D= dH - \sum_{i=1}^r E_i$ is nef on $X_{r}^n$ by the  argument given in the proof of Proposition  \ref{prop-verygeneralCurve}.  The $(n-1)$-fold self-intersection of the class is  $$D^{n-1} = d^{n-1}H_1- \sum_{i=1}^r E_{i,1}.$$ Since $r> 2 d^{n-1}$, this class is not in the span of lines. On the other hand, the class is pseudoeffective. A small perturbation of $D$ is ample. Hence, a sufficiently high multiple is very ample and the $(n-1)$-fold self-intersection is an effective curve. It follows that the class $D^{n-1}$ is pseudoeffective. 
\end{example}

\section{Non-finitely generated cones}\label{sec-nonfingen}

The cone of curves the blow-up of \(\P^2\) at \(10\) or more very general points is not entirely understood, and we will find it useful to assume the following standard conjecture.
\begin{conjecture}[Segre--Harbourne--Gimigliano--Hirschowitz (SHGH) conjecture, \cite{Gimigliano}]
\label{shghstatement}
Suppose that \(r \geq 10\) and that \(m_1 \geq m_2 \geq \cdots \geq m_{r}\) and  \(d > m_1 + m_2 +m_3\).  Then 
\[
H^0(X^2_{r},dH_1 - \sum_{i=1}^r m_i E_{i,1}) = \binom{d+2}{2} - \sum_{i=1}^r \binom {m_i+1}{2}
\]
\end{conjecture}

We next prove that the cone of codimension-\(2\) cycles on \(X^n_r\) is not finitely generated for \(r \geq n+6\), assuming the SHGH conjecture.  The calculation relies on the following observation of de Fernex.

\begin{theorem}[{\cite[Prop.\ 3.4]{DF}}]
\label{defernexthm}
Assume the SHGH conjecture holds for \(10\) points.  Let \(P \subset N^1(X^2_{10})\) be the positive cone
\[
P = \set{ D \in N^1(X^2_{10}) : D^2 \geq 0, D \cdot H \geq 0},
\]
where \(H\) is an ample divisor. Then
\[
\Effb_1(X^2_{10}) \cap K_{\geq 0} = P \cap K_{\geq 0}.
\]
\end{theorem}

Let \(Q \subset X^3_9\) be the strict transform of the unique quadric passing through the \(9\) points. Note that \(Q\) is isomorphic to \(X^2_{10}\), and so Conjecture~\ref{shghstatement} provides some information about the cone \(\Effb_1(Q)\).  However, the map \(N_1(Q) \to N_1(X^3_9)\) is not injective, since the two rulings of the quadric both map to the class of a line in \(\P^3\). The next lemma gives a criterion to show that certain extremal rays on \(\Effb_1(Q)\) nevertheless push forward to extremal rays on \(\Effb_1(X^3_9)\). 

Write \(r_1\) and \(r_2\) for the classes of the two rulings on the quadric, and let \(f_i = E_i \vert_Q\) be the exceptional curves.  Let \(\ell_{ij} = r_1 - f_i - f_j \in N_1(Q)\); this class is not effective on \(Q\), but \(i_\ast \ell_{ij}\) is effective in \(N_1(X^3_9)\) since it is the class of a line through the points \(p_i\) and \(p_j\).

\begin{theorem}
\label{extremalpushfwd}
Suppose that \(D\) is a class in \(N_1(Q)\) which satisfies:
\begin{enumerate}
\item \(D\) is nef, and if \(\gamma \in \Effb_1(Q)\) has \(D \cdot \gamma = 0\), then \(\gamma\) is a multiple of \(D\);
\item \(D \cdot K_Q > 0\);
\item \(D \cdot r_1 = D \cdot r_2\);
\item \(D \cdot \ell_{ij} > 0\) for all \(i\) and \(j\);
\end{enumerate}
Then \(i_\ast D\) is extremal on \(\Effb_1(X^3_9)\).
\end{theorem}
\begin{proof}
We claim first that \(D\) lies on a two-dimensional extremal face of the cone
\[
\Sigma = \Effb_1(Q) + \sum_{i,j} \R_{\geq 0} [\ell_{ij}] + \R [r_1-r_2] \subset N_1(Q).
\]
More precisely, if \(D = \alpha + \beta\) with \(\alpha, \beta \in \Sigma\), then $$\alpha = a_1 D + b_1 (r_1-r_2) \quad \mbox{and} \quad \beta = a_2 D + b_2 (r_1-r_2),$$ where $a_1$ and $a_2$ are positive.
Note that \(D\) is nef, \(D \cdot (r_1-r_2) = 0\), \(D \cdot \ell_{ij} > 0\) for all \(i,j\), and \(D \cdot f_k > 0\). Hence \(D\) is contained in the dual cone of \(\Sigma\).  By conditions (1) and (3), the classes in \(\Sigma\) with \(D \cdot C = 0\) are precisely \(\R_{\geq 0} D + \R(r_1-r_2)\).

We claim next that
\[
\Effb_1(X^3_9) = i_\ast \Effb_1(Q) + \sum_{i,j} \R_{\geq 0} i_\ast [\ell_{ij}] = i_\ast \Sigma.
\]
Suppose that \(\Gamma\) is an irreducible effective cycle on \(\Effb_1(X^3_9)\).  If \(Q \cdot \Gamma < 0\), then \(\Gamma\) must be contained in \(Q\), and so \([\Gamma]\) is contained in \(i_\ast \Effb_1(Q)\).  If \(Q \cdot \Gamma \geq 0\), then \(\Gamma\) satisfies \(2a \geq \sum_i b_i\), which means that \([\Gamma]\) is in the span of classes of lines \(i_\ast[\ell_{ij}]\) and lines \(i_\ast[f_k]\) in the exceptional divisors, by Lemma~\ref{maininequality}.  Each \(f_k\) is numerically equivalent to a curve in the quadric.

Suppose now that \(i_\ast D = \alpha + \beta\), where \(\alpha\) and \(\beta\) are pseudoeffective classes on \(\Effb_1(X^3_9)\).  Using the decomposition above, we can write \(\alpha = i_\ast \alpha_Q + \sum c_{ij} i_\ast [\ell_{ij}] \) and \(\beta = i_\ast \beta_Q + \sum d_{ij} i_\ast [\ell_{ij}]\), where \(\alpha_Q\) and \(\beta_Q\) are classes in \(\Effb_1(Q)\).  Because \(D \cdot K_Q > 0\) while \(K_Q \cdot \ell_{ij} = 0\), it must be that \(\alpha_Q \neq 0\).

We claim now that 
\[
D = \alpha_Q + \beta_Q + \sum (c_{ij} + d_{ij}) \ell_{ij} + f(r_1 - r_2)
\]
for some constant \(f\). Indeed, the two sides differ by an element of  the kernel of \(i_\ast: N_1(Q) \to N_1(X^3_9)\), which is generated by \(r_1 - r_2\), giving rise to the constant \(f\).

Since \(D^2 = 0\), \(D \cdot (r_1 - e_i -e_j ) > 0\) for any \(i\) and \(j\), and \(D \cdot (r_1 - r_2) = 0\), it must be that \(c_{ij} = d_{ij} = 0\) for all \(i,j\) and all \(k\).  We conclude that  
\[
\alpha_Q = a_1 D + b_1 (r_1 - r_2)\quad \mbox{and} \quad \beta_Q = a_2  D + b_2 (r_1 - r_2).
\]
Hence \(\alpha = i_\ast \alpha_Q = a_1 i_\ast D\) and \(\beta = i_\ast \beta_Q = a_2 i_\ast D\).  This shows that \(i_\ast D\) is extremal.
\end{proof}

The requirement that \(K_Q \cdot D > 0\) makes it tricky to explicitly exhibit such classes, since the \(K_Q\)-positive part of \(\Effb_1(Q)\) is difficult to determine without assuming the SHGH conjecture.

\begin{theorem}
Assume that the SHGH conjecture holds for blow-ups of \(\P^2\) at \(10\) very general points.  Then there exist infinitely many classes \(D\) satisfying the hypotheses of Theorem~\ref{extremalpushfwd}.
\end{theorem}
\begin{proof}
It is convenient to fix an identification \(Q \cong X^2_{10}\) and rewrite the hypotheses in the basis for \(N^1(Q)\) arising from this identification.
The strict transforms of the two rulings through the point \(p_1\) give disjoint \((-1)\)-curves on \(Q\), and these can be contracted.  The other \(8\) exceptional curves \(f_i\) can then be contracted to give a map to \(\P^2\).  Let \(e_0\) and \(e_1\) be the first two \((-1)\)-curves contracted, and let \(e_j = f_j\) for \(2 \leq j \leq 8\).  With respect to this new basis, we have  \(r_1 - f_1 = e_0\) and \(r_2 - f_1 = e_1\), and \(f_1 = h - e_0 - e_1\), where $h$ denotes the class of a line on $\P^2$.  

While the first two conditions in Theorem \ref{extremalpushfwd}  are independent of the basis, the last two can be rewritten as
\begin{enumerate}
\setcounter{enumi}{2}
\renewcommand{\theenumi}{\arabic{enumi}$^\prime$}
\item \(D \cdot e_0= D \cdot e_1\).
\item \(D \cdot e_0 > D \cdot e_j\) for any \(j > 1\), and \(D \cdot (h - e_1 - e_i - e_j) > 0\) for any \(i,j > 1\).
\end{enumerate}
The first part of (4) arises when \(i = 1 < j\), while the second case is when \(1 < i < j\).

Fix any \( \frac{1}{\sqrt{10}} < \delta < \frac{1}{3}\), and let \(\delta^\prime = \sqrt{ \frac{1- 2 \delta^2}{8}}\).  Observe that \(\frac{3}{10} < \frac{1}{6} \sqrt{\frac{7}{2}} < \delta^\prime < \delta\) for \(\delta\) in this range.
 Consider the divisor
\[
D_\delta = h - \delta (e_0 + e_1) - \delta^\prime \sum_{j=2}^9 e_j.
\]

We check each of the hypotheses in turn. To simplify notation, for the rest of this proof set $X = X_{10}^2$.
\begin{enumerate}[leftmargin=0in, itemindent=\dimexpr\parindent+\labelwidth+\labelsep\relax]
\item First we check that \(D_\delta\) is nef.  The cone theorem implies that 
\[
\Effb_1(X) = \Effb_1(X)_{K_X \geq 0} + \sum_i \R_{\geq 0} [C_i],
\]
where the \(C_i\) are \(K_X\)-negative curves. According to Theorem ~\ref{defernexthm},
\[
\Effb_1(X) \cap \Effb_1(X)_{K_{X} \geq 0} = P \cap K_{\geq 0}.
\]
Hence, it suffices to show that \(D_\delta \cdot C \geq 0\) if \(C\) is \(K_X\)-negative, and that \(D_\delta \cdot C \geq 0\) if \(C\) has \(C^2 \geq 0\) and \(C \cdot H > 0\).

First, suppose that \(C\) is a pseudoeffective class with \(K_X \cdot C < 0\).  We have
\[
3 D_\delta - K_X = (3\delta-1) (e_0 + e_1) + (3\delta^\prime-1) \sum_{j=2}^9 e_j,
\]
and so
\begin{align*}
3D_\delta \cdot C = K_X \cdot C + \left( (3\delta-1) (e_0 + e_1) + (3\delta^\prime-1) \sum_{j=2}^9 e_j \right) \cdot C
\end{align*}
However, since \(\delta < 1/3\), the number \(3 \delta - 1\) is negative.  It is easy to check that \(\delta^\prime < 1/3\) as well, and so the divisor on the right is a sum of exceptional divisors with negative coefficients.  If \(C\) is any curve other than one of the \(e_i\), then both terms on the right are negative.  If \(C\) is one of the curves \(e_i\), then \(D_\delta \cdot C > 0\) because \(\delta\) and \(\delta^\prime\) are both positive.

It remains to check that \(D_\delta \cdot C > 0\) if \(C\) is a class with positive self-intersection.  This follows from the Cauchy-Schwartz inequality: suppose that \(d^2 \geq \sum a_i^2\) and \(e^2 \geq \sum b_i^2\).  Then \(de \geq \sum a_i b_i\).  Moreover, equality is achieved if and only if \(C\) is a multiple of \(D_\delta\).

\item We have
\[
D_\delta \cdot K_Q = -3 + 2 \delta + 8 \delta^\prime > -3 + 10 \delta^\prime > 0,
\]
since \(\delta^\prime > 3/10\).
\setcounter{enumi}{2}
\renewcommand{\theenumi}{\arabic{enumi}$^\prime$}
\item Since \(D_\delta\) is of the form \(h - \delta e_0 - \delta e_1 - \cdots\), we have \(D \cdot e_0 = D \cdot e_1\).

\item Because \(\delta > \delta^\prime\), we have \(D \cdot e_0 > D \cdot e_j \) for any \(j > 1\).  We also have \(D \cdot (h - e_1 - e_i -e_j) = 1 - \delta - 2 \delta^\prime > 0\), since \(\delta^\prime < \delta < 1/3\).

\end{enumerate}
\end{proof}

\begin{remark}
One can even arrange that \(D_\delta\) is a rational class through judicious choice of \(\delta\).   For example,
\[
D_{\frac{226}{692}} = h - \frac{226}{692} (e_0 + e_1) - \frac{217}{692} \sum_{i=2}^9 e_i.
\]
However, in general such classes are not expected to have any effective representatives.
\end{remark}

Assuming the SHGH conjecture, we can now conclude that \(\Effb^2(X_r^n)\) is not finitely generated if \(r \geq n+6\). We need the following lemma, which guarantees that cones over extremal classes in \(\Effb_{k}(X_{r}^{n})\) are extremal in \(\Effb_{k+1}(X_{r+1}^{n+1})\).

\begin{lemma}\label{lem-cones}
Suppose that \(D = aH_k - \sum_{i=1}^r a_i E_{i,k}\) spans an extremal ray on \(\Effb_k(X_r^n)\).  Then \(CD = aH_{k+1} - aE_{0, k+1} - \sum_{i=1}^r a_i E_{i, k+1}\) spans an extremal ray on \(\Effb_{k+1}(X_{r+1}^{n+1})\).  In particular, if \(\Effb_k(X_r^n)\) has infinitely many extremal rays, then so does \(\Effb_{k+1}(X_{r+1}^{n+1})\).
\end{lemma}

Lemma \ref{lem-cones} immediately implies the following.

\begin{corollary}\label{cor-codim2infinite}
Assume the SHGH Conjecture for the blow-up of $\P^2$ at $10$ points. Then $\Effb^2(X_{r}^n)$ is not finitely generated if $r \geq n+6$.
\end{corollary}

\begin{proof}[Proof of Lemma \ref{lem-cones}]
Given $r+1$ very general points $p_0, \dots, p_r$ in $\P^{n+1}$, their projection from $p_0$ give $r$ very general points in $\P^n$. Let $D_i$ be effective cycles arbitrarily close to $D$ in $\Effb_k(X_r^n)$. Then the classes of the cones $CD_i$ over $D_i$ converge to $CD$. Hence, $CD \in \Effb_{k+1}(X_{r+1}^{n+1})$. 

Conversely, we claim that if \(CD = aH_{k+1} - aE_0 - \sum_{i=1}^r b_i E_{i, k+1}\) is a  pseudoeffective \((k+1)\)-cycle on \(X^{r+1}_{n+1}\), then \(D = aH_k -  \sum_{i=1}^r b_i E_{i, k}\) is a  pseudoeffective \(k\)-cycle on \(X^{r}_{n}\).  The class \(CD + \epsilon H_{k+1}\) is effective for any \(\epsilon > 0\). Let \(V_\epsilon\) be a (rational) cycle representing the class \(CD + \epsilon H_{k+1}\).  Let \(\ell_{0j}\) denote the strict transform on \(X_{n+1}^{r+1}\) of the line through \(p_0\) and \(p_j\). By Lemma \ref{lem-multiplicityline}, $V_{\epsilon}$ contains the line $\ell_{0j}$ with multiplicity  $\beta_j \geq b_j - \epsilon$. Let  $L \subset X^{n+1}_{r+1}$ be the proper transform of a general hyperplane in $\P^{n+1}$.  The lines $l_{0j}$ intersect $L$ in $r$ very general points $p$. The proper transform of the intersection 
 \(L \cap V_\epsilon\) gives an effective cycle with class \((a+\epsilon) H_k - \sum_{i=1}^j \beta_j E_{j,k}\).  Letting \(\epsilon\) tend to \(0\), we see that \(aH_k - \sum_{i=1}^r b_i E_{i,k}\) is pseudoeffective in $X_r^n$, as required.

Now, suppose that \(D = aH_k - \sum_{i=1}^r b_i E_{i,k}\) spans an extremal ray of \(\Effb_k(X^n_r)\).  We claim that $CD = aH_{k+1} - aE_{0, k+1} - \sum_{i=1}^r b_i E_{i, k+1}$ spans an extremal ray of \(\Effb_{k+1}(X^{n+1}_{r+1})\). Suppose \(CD = \alpha + \beta\), where \(\alpha\) and \(\beta\) are both pseudoeffective \((k+1)\)-cycles on \(X_{r+1}^{n+1}\).  Since any pseudoeffective class has \(a \geq b_0\), it must be that
\[
\alpha = cH_{k+1} - cE_{0, k+1} - \sum_{i=1}^r c_i E_{i, k+1}, \; \beta = dH_{k+1}- dE_{0, k+1} - \sum_{i=1}^r d_i E_{i, k+1}.
\]
Then
\[
\alpha_0 = cH_k - \sum_{i=1}^r c_i E_{i,k}, \; \beta_0 = dH_k - \sum_{i=1}^r e_i E_{i,k}.
\]
are pseudoeffective on $X_{r}^n$. Hence, $\alpha_0$ and $\beta_0$ are proportional to $D$. It follows that $\alpha$ and $\beta$ are proportional to $CD$ and $CD$ is extremal.
\end{proof}

There are several interesting remaining questions concerning the finite generation of cones of higher codimension.

\begin{question}
Can one show that $\Effb_{n-2}(X_r^n)$ is not finitely generated for $r \geq n+6$ independently of the SHGH Conjecture?
\end{question}

\begin{question}\label{ques-kinfinite}
Fix \(n\) and \(k\). Does there exist \(r\) for which \(\Effb_k(X^n_r)\) is not finitely generated?  How does \(r\) depend on \(n\) and \(k\)?
\end{question}

In particular, we have the following fundamental question.
\begin{question}\label{ques-1infinite}
For every $n$, does there exist $r$ for which $\Effb_1(X_r^n)$ is not finitely generated?
\end{question}

If $\Effb_1(X_r^n)$ is not finitely generated for $r \geq r_0$, then by Lemma \ref{lem-cones} $\Effb_k(X_{r+k-1}^{n+k-1})$ is not finitely generated for $r \geq r_0$. Hence, an affirmative answer to Question \ref{ques-1infinite} implies an affirmative answer to Question \ref{ques-kinfinite}.

\section{Blow-Ups at points in special position}\label{sec-special}

Until now we have considered blow-ups of $\P^n$ at linearly general or very general points.  It is also interesting to consider cones of effective cycles on blow-ups of $\P^n$ at special configurations of points.  The dependence of the cones on the position of the points can be subtle, which makes degeneration arguments difficult. We will see that the property of the effective cone being finite is neither an open nor closed condition, even in families where the vector space of numerical classes of $k$-dimensional cycles  has constant dimension.

\begin{proposition}\label{prop-linearspan}
Let $\Gamma$ be a set of $r$ points whose span is $\P^m \subset \P^n$. Let $X_{\Gamma}^n$ and $X_{\Gamma}^m$ denote the blow-up of $\P^n$ and $\P^m$ along $\Gamma$, respectively. 
\begin{enumerate}
\item Then $\Effb_k(X_{\Gamma}^n)$ is linearly generated for $m \leq k \leq n-1$ and 
\item $\Effb_k(X_{\Gamma}^n) = \Effb_k(X_{\Gamma}^m)$ for $k < m$.
\end{enumerate}
\end{proposition}
\begin{proof}
Since $X_{\Gamma}^m$ embeds in $X_{\Gamma}^n$ as the proper transform of the $\P^m$ spanned by $\Gamma$, any effective cycle $Z \subset X_{\Gamma}^m$ is also an effective cycle in $X_{\Gamma}^n$ with the same class. Hence, $\Effb_k(X_{\Gamma}^m) \subseteq \Effb_k(X_{\Gamma}^n)$ for $k< m$. Conversely, suppose that $k<m$. Let $Z$ be an effective cycle in $\P^n$ of dimension $k$. We may assume that $Z$ is not contained in an exceptional divisor. Choose a general point $p$. Let $q_i$ denote the projection of $p_i$ form $p$ and let $Z'$ be the projection of $Z$ from $p$. Then $Z'$ and $Z$ have the same degree and the multiplicities of $Z'$ at $q_i$ are greater than or equal to the multiplicities of $Z$  at $p_i$. Repeatedly projecting $Z$ to $\P^m$ from general points, we obtain an effective cycle contained in $\P^m$ with the same class. Taking closures, we obtain the reverse inclusion $\Effb_k(X_{\Gamma}^n) \subseteq \Effb_k(X_{\Gamma}^m)$. 

If $k \geq m$, let $L$ be a $k$-dimensional linear space containing $\Gamma$. Then the proper transform $\overline{L}$ of $L$ has class $H_k - \sum_{i=1}^r E_{i,k}$. Since a $k$-dimensional  variety  not contained in an exceptional divisor has class $aH_k - \sum_{i=1}^r b_i E_{i,k}$ with $a \geq b_i \geq 0$, we conclude that any $k$-dimensional effective cycle is a nonnegative linear combination of $[\overline{L}]$ and $E_{i,k}$, $1\leq i \leq r$. 
\end{proof}

By taking $m=1$, we obtain the following corollary.
\begin{corollary}
Suppose $\Gamma$ is a set of $r$ collinear points in $\P^n$. Then $\Effb_k(X_{\Gamma}^n)$ is linearly generated for every $1 \leq k \leq n-1$.
\end{corollary}
\begin{remark}
It was shown in \cite{Ottem} that the blow-up of $\P^2$ in collinear points is a Mori Dream Space (and indeed its Cox ring can be computed)2.  Consequently, the cone of curves and divisors are finite polyhedral. The previous corollary generalizes this to all cycles. 
\end{remark}

Let $L \cong \P^{n-1}$ be a hyperplane in $\P^n$. Let $\Gamma' \subset L$ be a set of points $p_1, \dots, p_r$ and let $p_0 \in \P^n$ be a point not contained in $L$. Let $\Gamma = \Gamma' \cup \{p_0\}$. Let $X_{\Gamma}^n$ and $X_{\Gamma'}^{n-1}$ denote the blow-up of $\P^n$ and $\P^{n-1}$ along $\Gamma$ and $\Gamma'$, respectively. Taking cones with vertex at $p_0$, we generate a subcone $\CEffb_k(X_{\Gamma'}^{n-1}) \subset \Effb_{k+1}(X_{\Gamma}^n).$  

\begin{proposition}\label{prop-addpoint}
 The cone $\Effb_{k+1}(X_{\Gamma}^n)$ is generated by $\CEffb_k(X_{\Gamma'}^{n-1})$,  $\Effb_{k+1}(X_{\Gamma'}^{n-1})$ and $E_{0, k+1}$. Furthermore, the extremal rays of $\CEffb_k(X_{\Gamma'}^{n-1})$ are also extremal in $\Effb_{k+1}(X_{\Gamma}^n)$.
\end{proposition}
\begin{proof}
Let $Z= aH_{k+1} - \sum_{i=0}^{r} b_i E_{i, k+1}$ be an irreducible $(k+1)$-dimensional variety in $X_{\Gamma}^n$. We may assume that $Z$ is not contained in any exceptional divisors. Otherwise, its class is a positive multiple of $E_{i, k+1}$. The proper transform of $L$ in $X_{\Gamma}^n$ is isomorphic to $X_{\Gamma '}^{n-1}$. If $Z$ is contained in $X_{\Gamma'}^{n-1}$, then the class of $Z$ is in $\Effb_{k+1}(X_{\Gamma'}^{n-1})$. Otherwise, $Z \cap X_{\Gamma'}^{n-1}$ is an effective $k$ cycle with class $\alpha = aH_{k} - \sum_{i=1}^{r} b_i E_{i, k}$. Then the cone $C(\alpha)$ is an effective class in $X_{\Gamma}^n$ and since $b_0 \leq a$, $[Z]$ is in the span of  $E_{0, k+1}$ and $C(\alpha)$. 

Let $Z$ be a cycle that generates an extremal ray of $\CEffb_k(X_{\Gamma'}^{n-1})$. Suppose $[Z]= \alpha + \beta$ in $\Effb_{k+1}(X_{\Gamma}^n)$. Since $b_0 \leq a$ holds on $\Effb_{k+1}(X_{\Gamma}^n)$ and $b_0=a$ on $\CEffb_k(X_{\Gamma'}^{n-1})$, we must have that both $\alpha$ and $\beta$ satisfy $b_0 = a$. We can perturb $\alpha$ and $\beta$ by $\epsilon H_{k+1}$ to obtain rational effective classes. Since the coefficient of $E_{0,k+1}$ of any class contained in $\Effb_{k+1}(X_{\Gamma'}^{n-1})$ is $0$, the coefficients of any component of the class contained in $\Effb_{k+1}(X_{\Gamma'}^{n-1})$ are bounded by $\epsilon$.  Taking cones over the classes of the hyperplane sections of the remaining subvarieties and letting $\epsilon$ tend to zero, we see that both $\alpha$ and $\beta$ are contained in $\CEffb_k(X_{\Gamma'}^{n-1})$. By the extremality of $Z$, we conclude that they are both proportional to $[Z]$.
\end{proof}

\begin{corollary}\label{cor-notfinite}
Let $\Gamma$ be a set of points $\{q_1, \dots, q_9, p_1, \dots, p_{s-1}\}$ such that $q_1, \dots, q_9$ are general points in a plane $P \subset \P^n$ and $p_1, \dots, p_{s-1}$ are linearly general points with span disjoint from $P$. Then $\Effb_k(X_{\Gamma}^n)$ is not finitely generated for $k \leq s$ and linearly generated for $k>s$. 
\end{corollary}
\begin{proof}
When $r \geq 9$, the blow-up of $\P^2$ at $r$ general points has infinitely many $(-1)$-curves, which span extremal rays of the effective cone of curves. Applying Proposition \ref{prop-addpoint} $(k-1)$-times, the cones over the classes of $(-1)$-curves with vertex $p_1, \dots, p_{k-1}$ provide infinitely extremal rays of $\Effb_k(X_{\Gamma}^n)$ for  $k \leq s$. The linear generation of $\Effb_k(X_{\Gamma}^n)$ for $k> s$ follows from Proposition \ref{prop-linearspan}
\end{proof}

\begin{corollary}\label{cor-closed}
\begin{enumerate}
\item Linear generation of $\Effb_k(X_{\Gamma}^n)$ is not closed in smooth families.
\item Finite generation of $\Effb_k(X_{\Gamma}^n)$ is not closed in smooth families.
\end{enumerate}
\end{corollary}
\begin{proof}
Let $n \geq k+8$. Take a  general smooth curve $B$ in $(\P^n)^{k+8}$ which avoids all the diagonals and contains a point $0\in B$ where $9$ of the points are general points in a plane $P$ and the remaining points are in linearly general position with span not intersecting $P$. Such curves exist by Bertini's Theorem since the diagonals have codimension $n \geq 2$. Consider the family $\mathcal{X} \rightarrow B$, where $\mathcal{X}_b$ is the blow-up of $\P^n$ in the $k+8$ points $\Gamma_b$ parameterized by $b \in B$. If the points in $\Gamma_b$ are in linearly general position, then by Lemma \ref{toriccase} the cone $\Effb_k(X_{\Gamma_b})$ is linearly generated. In particular, the cone is finite. However, by Corollary \ref{cor-notfinite}, $\Effb_k(X_{\Gamma_0})$ is not finitely generated. In particular, the cone is not linearly generated.
\end{proof}

\begin{corollary}\label{cor-open}
\begin{enumerate}
\item Linear generation of $\Effb_k(X_{\Gamma}^n)$ is not open in smooth families.
\item Finite generation of $\Effb_k(X_{\Gamma}^n)$ is not open in smooth families.
\end{enumerate}
\end{corollary}
\begin{proof}
Let $B$ be a smooth curve parameterizing $9$ general points in a plane $P$ becoming collinear at $0 \in B$. Let $\Gamma'$ be $k-1$ points in general linear position in $\P^n$ whose span is disjoint from $P$. Let $\Gamma_b$ be the union of $\Gamma'$ and the points parameterized by $b$. Consider the family $\mathcal{X} \rightarrow B$ obtained by blowing up $\P^n$ along $\Gamma_b$. When the points are collinear, $\Effb_k(X_{\Gamma_0})$ is linearly generated. However, for the general point of $B$, $\Effb_k(X_{\Gamma_b})$ is not finitely generated by Corollary \ref{cor-notfinite}.   
\end{proof}

\begin{remark}
Corollary \ref{cor-open} is well-known for cones of divisors. For example, Castravet-Tevelev \cite{CastravetTevelev} prove that the blow-up of $\P^n$ at points on a rational normal curve is a Mori Dream Space. In particular, if we specialize a large number of points to lie on a rational normal curve, we see that being a Mori Dream Space is not an open condition. 
\end{remark}

One can ask for the finite/linear generation of $\Effb_k(X_{\Gamma}^n)$ for $\Gamma$ any special set of points. Perhaps the following question is the most interesting among them.

\begin{question}\label{ques-rnc}
Let $\Gamma$ be a set of points contained in a rational normal curve in $\P^n$. Is $\Effb_k(X_{\Gamma}^n)$ finitely generated? Is $\Effb_k(X_{\Gamma}^n)$  generated by the classes of cones over secant varieties of projections of the rational normal curve?
\end{question}
By results of Castravet and Tevelev, the answer to Question \ref{ques-rnc} is affirmative for curves and divisors. The cone of curves $\Effb_1(X_{\Gamma}^n)$ is generated by the class of the proper transform of the rational normal curve $n H_1 - \sum_{i=1}^r E_{i,1}$ and the classes of lines. The rational normal curve is cut out by quadrics. If a curve $B$ has positive intersection with a quadric containing the points, then by Lemma \ref{maininequality} the class of $B$ is spanned by the classes of lines. Otherwise, $B$ is contained in the base locus of all the quadrics containing the rational normal curve. Hence, $B$ is a multiple of the rational normal curve. 
Castravet and Tevelev show that the  classes of divisors are generated by linear spaces and codimension-$1$ cones over secant varieties of the projection of the rational normal curve \cite{CastravetTevelev}.  We do not know whether $\Effb_2(X_{\Gamma}^n)$ is generated by the classes of planes and cones over the rational normal curve with vertex a point of $\Gamma$.

\bibliographystyle{plain}

\end{document}